\documentclass[10pt]{article}
\usepackage[a4paper,
            left=1.25in,
            right=1.25in,
            top=1in,
            bottom=1in,
            footskip=.25in
]{geometry}

\usepackage[T1]{fontenc}
\usepackage[sc]{titlesec}
\usepackage{blindtext,color}
\usepackage{microtype}
\usepackage{scrextend}
\usepackage{upgreek}
\usepackage{stmaryrd}
\usepackage{tocloft}
\usepackage{algorithm}
\usepackage{algpseudocode}
\usepackage{amssymb}
\usepackage{amsmath}
\usepackage{amsthm}
\usepackage{amsfonts}
\usepackage{shadethm}
\usepackage{bbold}
\usepackage{mathtools}
\usepackage[table,dvipsnames]{xcolor}
\usepackage[colorinlistoftodos,prependcaption,textsize=tiny]{todonotes}
\usepackage{authblk}

\usepackage{xspace}
\usepackage{dsfont}
\usepackage{marginnote}
\usepackage[normalem]{ulem}
\usepackage{bm}
\usepackage{graphicx}
\usepackage[nottoc]{tocbibind}
\usepackage[citecolor=blue,colorlinks=true,linkcolor=blue]{hyperref}%
\usepackage[capitalise, noabbrev]{cleveref}
\crefname{equation}{eq.}{equations}
\Crefname{equation}{Eq.}{Equations}
\usepackage{tikz}
\usetikzlibrary{cd}
\usepackage{sectsty}
\usepackage{cancel}
\usepackage{booktabs,multirow}
\usepackage{csvsimple}
\providecommand{\keywords}[1]
{
  \small	
  \textbf{\textit{Keywords---}} #1
}

\date{\vspace{-5ex}}

\newcounter{mythm}

\newcounter{myrem}
\newcounter{myexa}
\newcounter{mypro}
\newcounter{mycol}

\newcounter{mydef}

\newtheorem{definition}[mydef]{Definition}
\newtheorem{remark}[myrem]{Remark}

\newtheorem{example}[myexa]{Example}
\newtheorem{proposition}[mypro]{Proposition}
\newtheorem{corollary}[mycol]{Corollary}

\newcommand{\Divergence}{\mathcal{D}}
\newcommand{\BregmanLogDet}{\mathcal{D}_\textnormal{LD}}

\newcommand{\inv}{^{-1}}
\newcommand{\invhalf}{^{-\half}}
\newcommand{\transp}{^\top}
\newcommand{\invtransp}{^{-\top}}
\newcommand{\Hermitian}{^*}
\newcommand{\invHermitian}{^{-*}}

\newcommand{\given}{\,|\,}
\newcommand{\trace}[1]{\Tr\big( #1 \big)}
\newcommand{\logdet}[1]{\log \det( #1 )}
\newcommand{\half}{{\frac{1}{2}}}

\newcommand{\Sn}{\mathbb{H}^n}

\newcommand{\Snp}{{\Sn_+}}
\newcommand{\Snpp}{{\Sn_{++}}}

\newcommand{\Realn}{\mathbb{R}^n}

\newcommand{\Rmn}{\mathbb{R}^{m\times n}}

\newcommand{\Rnk}{\mathbb{R}^{n\times k}}

\newcommand{\Cnr}{\mathbb{C}^{n\times r}}
\newcommand{\Crr}{\mathbb{C}^{r\times r}}

\DeclareMathOperator{\diag}{diag}
\DeclareMathOperator{\Tr}{trace}
\DeclareMathOperator{\rank}{rank}

\DeclareMathOperator{\domain}{dom}
\DeclareMathOperator{\relint}{ri}

\DeclareMathOperator*{\minimise}{minimise}
\DeclareMathOperator*{\subto}{s.t.}
\newcommand{\MATLAB}{\textsc{Matlab}\xspace}

\newcommand{\MatVec}{\pi}

\makeatletter
\newsavebox{\@brx}
\newcommand{\llangle}[1][]{\savebox{\@brx}{\(\m@th{#1\langle}\)}%
  \mathopen{\copy\@brx\mkern2mu\kern-0.9\wd\@brx\usebox{\@brx}}}
\newcommand{\rrangle}[1][]{\savebox{\@brx}{\(\m@th{#1\rangle}\)}%
  \mathclose{\copy\@brx\mkern2mu\kern-0.9\wd\@brx\usebox{\@brx}}}
\makeatother

\newcommand*{\ldblbrace}{\{\mskip-5mu\{}
\newcommand*{\rdblbrace}{\}\mskip-5mu\}}

\newcommand{\BregmanIndexSet}[1]{\mathcal{I}_{ #1 }^\star}
\newcommand{\InverseBregmanIndexSet}[1]{\mathcal{J}_{ #1 }^\star}

\newcommand{\NysIndefLabel}{\text{Indef}}

\newcommand{\Nystrom}[1]{{#1}^\text{Nys}\langle\Omega\rangle}
\newcommand{\NystromIndefinite}[1]{{#1}^\NysIndefLabel\langle\Omega\rangle}

\newcommand{\tsvd}[2]{\llbracket #1 \rrbracket_{#2}}
\newcommand{\tsvdks}[2]{\llbracket #1 \rrbracket_{#2}^\textnormal{Kry}}

\newcommand{\BregTruncR}[2]{\llangle #1 \rrangle_{ #2 }}
\newcommand{\BregTrunc}[1]{\BregTruncR{ #1 }{r}}
\newcommand{\BregTruncRev}[1]{\ldblbrace #1 \rdblbrace_r}
\newcommand{\BregTruncPositive}[1]{\BregTrunc{ #1 }^\oplus}
\newcommand{\BregTruncNegative}[1]{\BregTrunc{ #1 }^\ominus}
\newcommand{\BregmanApproxAlpha}{\tilde E_r(\alpha)}
\newcommand{\BregmanApproxAlphaPositive}{\tilde E_{\rplus}(\alpha)}
\newcommand{\BregmanApproxAlphaNegative}{\tilde E_{\rminus}(\alpha)}

\newcommand{\rplus}{r^\oplus}
\newcommand{\rminus}{r^\ominus}
\newcommand{\rplusalpha}{\rplus(\alpha)}
\newcommand{\rminusalpha}{\rminus(\alpha)}

\newcommand{\PrecondSVD}[1]{{#1}_r^\textnormal{SVD}}
\newcommand{\PrecondSVDKS}[1]{{#1}_r^\textnormal{KrySVD}}
\newcommand{\PrecondBreg}[1]{{#1}_r^\textnormal{Br}}
\newcommand{\PrecondBregAlpha}[2]{{#1}_r^{\textnormal{Br}}({#2})}
\newcommand{\PrecondNys}[1]{{#1}_r^\textnormal{Nys}}
\newcommand{\PrecondNysIndef}[1]{{#1}_r^\NysIndefLabel}
\newcommand{\PrecondBregRev}[1]{{#1}_r^\textnormal{RBr}}

\newcommand{\CSVSVDKS}{$\PrecondSVDKS{S}$}

\newcommand{\CSVPrecondBregAlpha}[1]{\PrecondBregAlpha{S}{\alpha}}
\newcommand{\CSVPrecondNys}{\PrecondNys{S}}
\newcommand{\CSVPrecondNysIndef}{\PrecondNysIndef{S}}

\newcommand{\CSVICHOL}{\texttt{ichol}}

\algnewcommand\algorithmicoutput{\textbf{Output:}}
\algnewcommand\Output{\item[\algorithmicoutput]}

\DeclareUrlCommand\UScore{}
\newcommand{\expUScore}{%
  \expandafter\expandafter\expandafter
  \UScore
  \expandafter\expandafter\expandafter
}
\newtheorem{theorem}[mythm]{Theorem}

\title{
A New Matrix Truncation Method for Improving Approximate Factorisation Preconditioners
}

\author[1]{Andreas A. Bock\thanks{Corresponding author: \texttt{aasbo@dtu.dk}.}}
\author[1]{Martin S. Andersen\thanks{\texttt{mskan@dtu.dk}.}}
\affil[1]{Department of Applied Mathematics and Computer Science, Technical University of Denmark}

\begin{document}
\maketitle

\begin{abstract}
\noindent
In this experimental work, we present a general framework based on the Bregman log
determinant divergence for preconditioning
Hermitian positive definite linear systems. We explore this divergence
as a measure of discrepancy between a preconditioner and a matrix.
Given an approximate factorisation of a given matrix, the proposed framework 
informs the construction of a low-rank approximation of the typically indefinite 
factorisation error. The resulting preconditioner is therefore a sum of a Hermitian positive 
definite matrix given by an approximate factorisation plus a low-rank matrix.
Notably, the low-rank term is not generally obtained as a truncated singular value 
decomposition (TSVD). This framework leads to a new truncation where principal directions
are not based on the magnitude of the singular values, and we prove that such
truncations are minimisers of the aforementioned divergence.
We present several numerical examples showing that the proposed preconditioner 
can reduce the number of PCG iterations compared to a preconditioner constructed
using a TSVD for the same rank.
We also propose a heuristic to approximate the proposed preconditioner in the
case where exact truncations cannot be computed explicitly (e.g. in a large-scale
setting) and demonstrate its effectiveness over TSVD-based approaches.
\end{abstract}

\keywords{preconditioner, low-rank approximation, Bregman log determinant divergence.}

\section{Introduction}

We construct a preconditioner for the iterative solution of the following 
system:
\begin{equation}\label{eq:Sxb}    
Sx = b, \quad S \in \Snpp,
\end{equation}
where $\Snpp$ denotes the cone of  $n\times n$ Hermitian positive definite
matrices (HPD). While it may be too expensive to factorise $S$, there may
exist matrices
$Q$ such that $QQ\Hermitian \approx S$ that can be computed cheaply in
terms of both storage and computation time \cite{scott2023algorithms}.
Such a matrix can be used as a preconditioner, transforming \eqref{eq:Sxb} into 
the mathematically
equivalent problem
\begin{equation}\label{eq:Sxb:precond}
Q\inv S Q\invHermitian z  = Q\inv b, \quad Q\Hermitian x = z,
\end{equation}
for which the conjugate gradient (CG) method converges faster than
for \cref{eq:Sxb}. One such example when $S$ is sparse is an incomplete
Cholesky factorisation \cite{greenbaum1997iterative,golub2013matrix}, which
is used in many applications across scientific computing.
The generally indefinite matrix
\begin{equation}\label{eq:E}
E = S - QQ\Hermitian
\end{equation}
represents the approximation factorisation error, which we assume can be 
well-approximated by a low-rank matrix. The aim is therefore to improve 
approximate factorisation preconditioners by incorporating a low-rank 
approximation to $Q\inv EQ\invHermitian$.
In the context of low-rank matrix approximations, the tacit assumption is 
often that the objective is to find its dominant directions with respect 
to the spectral and Frobenius norms, thanks to the Young-Eckhart-Mirksy
theorem \cite{eckart1936approximation}. This leads to the well-known
\emph{truncated singular value decomposition} $\tsvd{X}{r}$ of $X$ to 
order $r$. For square matrices, the $\ell_2$ condition number
\[
\kappa_2(S) = \frac{\sigma_1(S)}{\sigma_n(S)},
\]
is also based on a notion of nearness in these norms. It is sometimes
used as a metric for evaluating the fitness of preconditioners due to 
its presences in the well-known upper bound for the error between the
true solution $x$ to \cref{eq:Sxb} and the $k\textsuperscript{th}$ 
iterate $x_k$ of CG \cite{greenbaum1997iterative}:
\[
\| x - x_k \|_S \leq 2\Big(\frac{\sqrt{\kappa_2(S)} -1}{\sqrt{\kappa_2(S)} +1}\Big)^k\|x - x_0\|_S,
\]
for some initial guess $x_0$.
In this paper, we study a different nearness measure for low-rank approximations,
namely the Bregman log determinant matrix divergence \cite{bregman1967relaxation}.
\begin{definition}\label{def:Bregman_divergence}
The Bregman matrix divergence $\Divergence_\phi:\domain \phi\times\relint \domain\phi\rightarrow [0, \infty)$
associated with a proper, continuously-differentiable, strictly convex seed function $\phi$ is defined as
\begin{align*}
    \Divergence_\phi(X,Y) & = \phi(X) - \phi(Y) - \langle \nabla\phi(Y), X - Y\rangle.
\end{align*}
The log determinant matrix divergence, induced by $\phi(x) = -\log\det X$, is given by
\[
\BregmanLogDet(X, Y) = \trace{XY\inv} - \logdet{XY\inv} - n,
\]
where in this case $\domain \phi = \Snpp$.
\end{definition}
A key property of the Bregman log determinant divergence
is its invariance to congruence transformations, i.e. for $X,Y\in\Snpp$ and 
invertible $Q$ we have
\begin{equation}\label{eq:Bregman:invariance}
\BregmanLogDet(X,Y) = \BregmanLogDet(Q X Q\Hermitian, Q Y Q\Hermitian).
\end{equation}
This invariance is central to our contribution in this paper. Indeed, we
can write the matrix $S$ from \cref{eq:Sxb} in terms of a given approximate
factorisation $QQ\Hermitian \approx S$ as follows:
\begin{equation}\label{eq:S_with_Error}
S = Q(I + \tilde E)Q\Hermitian, \quad\text{where}\quad \tilde E = Q\inv E Q\invHermitian.
\end{equation}
This leads us to study the candidate preconditioners of the form
\begin{equation}\label{eq:PC_form}
P = Q(I + W)Q\Hermitian\in\Snpp, \quad \rank W \leq r, \quad W\in \Sn,
\end{equation}
where $\Sn$ denote the space of $n\times n$ Hermitian matrices. Thanks to
\cref{eq:Bregman:invariance,eq:PC_form}, we therefore have
\begin{equation}
\BregmanLogDet(S, P) = \BregmanLogDet(I + \tilde E, I + W).
\label{eq:I+W}
\end{equation}
The aim is to determine the variable $W$ by minimising the divergence above
over a suitable set of low-rank matrices. It is not known which rank
$r$ leads to the right balance between acceleration of PCG and computational
cost (e.g. storage, application of the inverse of $P$) for a given setting. However,
the motivation for our work is that $Q$ may be inexpensive to compute (or 
already implemented), and if $\tilde E$ is approximately low-rank, then the 
preconditioners \cref{eq:PC_form} may be of interest in practice.\\

This leads to the main contribution of this work. We will show that such
minimisers are generally not given as a rank $r$ TSVD of the matrix 
$\tilde E$, giving rise to a new notion of a matrix truncation which we call
a \emph{Bregman log determinant (BLD) truncation}. We provide an 
explicit characterisation of the solution and highlight the effect of the 
indefiniteness of $\tilde E$ and how it differs from a TSVD. This approximation
is covered in \cref{sec:approximation_problem}, and the resulting proposed
preconditioner analysed in \cref{sec:preconditioner}.
\cref{sec:Bregman_directions} presents an approximation of the BLD truncation 
suitable for larger instances, or when access to $S$ is only given by the 
matrix-vector product $x \mapsto S x$. The associated approximate preconditioner
is discussed in \cref{sec:preconditioner:approx}.
\Cref{sec:numerical_results} contains numerical results, where the focus
is on experimenting with the truncations presented in this paper and comparing
PCG iteration for the associated preconditoners to a preconditioner based on the
TSVD. \Cref{sec:numerical_results:ichol} presents results where $Q$ stems
from an incomplete Cholesky factorisation of matrices $S$ from the SuiteSparse
Matrix Collection of modest size, where the truncations can be computed exactly.
In \cref{sec:numerical_results:large}, we consider larger instances where these
truncations cannot be computed cheaply to high accuracy, and instead compares the approximate 
preconditioner from \cref{sec:preconditioner:approx} with preconditioners with
low-rank approximations based on randomised approximations \cite{martinsson2020randomized}.
\Cref{sec:summary} contains a summary and discusses the need for a theoretical
investigation between the Bregman log determinant divergence and convergence of
iterative methods to support the empirical evidence shown in this paper.

\subsection{Related work}\label{sec:literature}

Much work has been dedicated to the construction of explicit preconditioners
based on minimisation of various quantities, the most well-known being the
Frobenius norm
\cite{benson1982iterative,gould1995approximate,benzi1999comparative}. The idea
is to minimise the \emph{stability} \cite{benzi2002preconditioning} 
\[
\| I - MA\|_F^2
\] 
(or $\| I - AM\|_F^2$ in the case of right-preconditioning) over some set of
inverse preconditioners $M$ with sparsity pattern $\mathcal{S}$. This expression
can be decomposed into a sum, leading to $n$ linear least-squares problems that
can be solved indepedently. It is difficult to prescribe a good sparsity pattern
$\mathcal{S}$, so adaptive approaches \cite{cosgrove1992approximate,chow1998approximate}
have been developed as a result, with the \emph{sparse approximate inverse} (SPAI)
\cite{grote1997parallel} being one of the most well-known.
Another functional of interest in measuring the quality of a preconditioner
is the so-called \emph{accuracy} of an approximate factorisation $\hat L \hat U
\approx A$ given by
\[
\| A - \hat L \hat U\|_F^2.
\]
This quantity alone can be predictive of required PCG iterations for some
matrices (e.g. Stieltjes matrices), but this is not true in general 
\cite{benzi2002preconditioning}, and it is not always clear which
functionals lead to good preconditioners. For symmetric matrices,
\cite{kolotilina1993factorized,kolotilina1995factorized} introduced the
factored approximate inverse preconditioner (FSPAI) also based on 
$\| I - MA\|_F$, but uses a factored form of $M=\hat L\hat L\transp$ (with
prescribed sparsity pattern) to preserve symmetry 
for use with the conjugate gradient method.
\cite{kaporin1990preconditioning,kaporin1994new} take a 
similar approach using a different functional as objective.
We refer to the recent work \cite{scott2023algorithms} for overviews of relevant 
procedures, in particular sparse and incomplete factorisations.\\

The design of preconditioners using a combination of approximate factorisations
and low-rank approximations has been studied before in \cite{higham2019new}.
In this reference the authors assume that $S$ is a general $n\times n$ non-singular
matrix and that an approximate LU factorisation is available of the form
\[
S = L U + \Delta S.
\]
The authors investigate the numerical rank of the error matrix
\[
E = U\inv L\inv \Delta S - I,
\]
and propose the preconditioner $(I + \tsvd{E}{r})U\inv L\inv$. Although we 
limit our attention to Hermitian matrices, the work here aims to leverage the
same  observation.\\

Bregman divergences 
\cite{bregman1967relaxation,amari2010information,amari2016information} have
been studied in connection with matrix approximations in numerous contexts 
including general matrix nearness problems \cite{dhillon2008matrix}, 
non-negative matrix factorisations 
\cite{sra2005generalized}, computational finance \cite{nock2012mining},
sparse inverse covariance estimation \cite{bollhofer2019large} and machine
learning \cite{banerjee2005clustering,pmlr-v119-cilingir20a,kulis2009low}.
The present work is most similar to \cite{schafer2021sparse}, where the 
authors find a sparse inverse Cholesky factorisation \cite{schafer2021sparse}
of the covariance matrix $\Theta$ by minimising the Kullback--Leibler divergence
\[
D_\mathrm{KL}\big(\Theta, (LL\transp)\inv\big)
\]
over a set of factors $L$.
For finite-dimensional Gaussian densities, this divergence is exactly the
Bregman log determinant divergence in \cref{def:Bregman_divergence}. We also
seek closeness in this sense, but for the purposes of finding a preconditioner,
and we parameterise the variable using $Q$ and a low-rank matrix (cf.
\cref{eq:PC_form}).\\

In a previous paper \cite{bock2023preconditioner}, the authors studied \cref{eq:Sxb}
where $S$ was assumed to be of the form
\begin{equation}\label{eq:Sxb:PSD}
S = Q(I + G) Q\Hermitian,
\end{equation}
where $G$ was assumed to be Hermitian \emph{positive semi-definite}
(denoted by $\Snp$). Problems with such structure arise in, for example, 
variational data assimilation \cite{tabeart2021saddle} where $QQ\Hermitian$
is a Cholesky factorisation of a matrix corresponding to a model and $G$
stems from an observation operator. Given the structure \cref{eq:Sxb:PSD}, the
authors studied preconditioners of the form
\[
Q(I + X)Q\Hermitian, \quad X \in \Snp,\quad \rank{X} \leq r,
\]
and showed that the preconditioner
\begin{equation}\label{eq:scaled}
\PrecondSVD{S} = Q(I + \tsvd{G}{r})Q\Hermitian
\end{equation}
is a minimiser of Bregman log determinant divergence to $S$.
Further, it was shown that when $G$ does not have full rank, 
$\PrecondSVD{S}$ minimises the condition number
$\kappa_2({\PrecondSVD{S}}\invhalf S {\PrecondSVD{S}}\invhalf)$
among all such candidate preconditioners. However,
this result only applies when $G$ is Hermitian positive semi-definite 
with non-trivial kernel. The efforts here generalise the setting in 
\cite{bock2023preconditioner} in the following sense. $QQ\Hermitian$
is taken to be an approximate factorisation of $S$, and writing $S$ 
of the form in \cref{eq:S_with_Error} results in the term $\tilde E$,
which is indefinite in general, leading to the study of the 
preconditioners \cref{eq:PC_form}. \Cref{eq:scaled} corresponds to the 
case where $\tilde E$ is positive semi-definite.
\section{Approximation problem}\label{sec:approximation_problem}

In this section we compute a low-rank approximation to $\tilde E$ in the
sense of the Bregman log determinant divergence. We write
\begin{subequations}\label{eq:G_and_W}
\begin{align}
\tilde E & = V \Theta V\Hermitian = \begin{bmatrix} v_1|\ldots | v_n\end{bmatrix} \begin{bmatrix}
\theta_1 & & \\
& \ddots & \\
& & \theta_n
\end{bmatrix}\begin{bmatrix} v_1|\ldots | v_n\end{bmatrix}\Hermitian,\label{eq:G_and_W:G}\\
W & = U \Lambda U\Hermitian =
\begin{bmatrix} u_1|\ldots | u_n\end{bmatrix} \begin{bmatrix}
\lambda_1 & & \\
& \ddots & \\
& & \lambda_n
\end{bmatrix}\begin{bmatrix} u_1|\ldots | u_n\end{bmatrix}\Hermitian,\label{eq:G_and_W:W}
\end{align}
\end{subequations}
where we assume $U$ and $V$ are unitary. In view of \cref{eq:I+W}, the approximation
problem that we are interested in is
\begin{subequations}\label{eq:Bregman_truncation:problem}
\begin{align}
\minimise_{W\in\Sn}\quad & \BregmanLogDet(I+\tilde E, I+W)\\
\subto\quad
& I + W \in \Snpp,\\
& \rank(W) \leq r.
\end{align}
\end{subequations}

We can write the Bregman log determinant divergence in terms of the
decompositions in \cref{eq:G_and_W}:
\begin{equation}\label{eq:Divergence:general}
\BregmanLogDet(I + \tilde E, I + W) = \sum_{i=1}^n\sum_{j=1}^n (u_i\Hermitian v_j)^2 \Big( \frac{1 + \theta_j}{1 + \lambda_i} - \log \big(\frac{1 + \theta_j}{1 + \lambda_i}\big) - 1\Big).
\end{equation}
We note that the $n\times n$ matrix $T$ with entries
\begin{equation}\label{eq:bistochastic}
T_{ij} = (u_i\Hermitian v_j)^2, \quad i, j = 1,\ldots,n,
\end{equation}
is a bistochastic matrix. By assumption, 
$S \in \Snpp$, so the eigenvalues of $I+\tilde E$ are strictly positive 
which, in turn, confides  the eigenvalues of $\tilde E$ to the interval
\begin{equation}\label{eq:eigenvalue_interval}
\theta_j \in (-1, \infty),\quad j=1, \ldots, n.
\end{equation}

By setting $U=V$ as an ansatz in \cref{eq:G_and_W}, we reduce the problem in
\cref{eq:Bregman_truncation:problem} to finding $r$ non-zero entries of
the diagonal matrix $\Lambda$ that minimise the objective. Indeed, for this ansatz,
$T$ is the identity matrix and \cref{eq:Divergence:general} becomes
\begin{equation}\label{eq:Divergence:U=V}
\BregmanLogDet(I + \tilde E, I + W) = \sum_{i=1}^n \Big( \frac{1 + \theta_i}{1 + \lambda_i} - \log \big(\frac{1 + \theta_i}{1 + \lambda_i}\big) - 1\Big).
\end{equation}
We have the constraint that only $r<n$ of the elements of $\lambda$ can
be non-zero. If $\lambda_i = 0$
\[
\Big[\frac{1 + \theta_i}{1 + \lambda_i} - \log \big(\frac{1 + \theta_i}{1 + \lambda_i}\big) - 1\Big]_{\lambda_i = 0} = \theta_i - \log (1 + \theta_i).
\]
We can therefore pose \eqref{eq:Bregman_truncation:problem} as a 
combinatorial problem:
\begin{subequations}\label{eq:Divergence:U=V:combinatorical}
\begin{align}
\min_{d\in \{0, 1\}^n} \quad & \sum_{i=1}^n \Big(\theta_i - \log (1 + \theta_i)\Big) (1 - d_i)\label{eq:Divergence:U=V:combinatorical:objective}\\
\subto\quad & \sum_{i=1}^n d_i \leq r.
\end{align}
\end{subequations}
Problem \eqref{eq:Divergence:U=V:combinatorical} 
is equivalent to finding the $r$ largest values of
\[
\theta_i - \log\big(1 + \theta_i\big),\quad i=1,\ldots,n.
\]
We recover the eigenvalues of a minimiser of \cref{eq:Divergence:U=V} from
a solution to \cref{eq:Divergence:U=V:combinatorical}:
\begin{equation}\label{eq:U=V:lambda_optimal}
\lambda_i^\star = \begin{cases}
\theta_i & \text{if}\, d_i = 1,\\
0 & \text{otherwise}.
\end{cases}
\end{equation}
\Cref{eq:U=V:lambda_optimal} tells us that the
non-zero eigenvalues of a minimiser of \cref{eq:Divergence:U=V} are exactly 
those of $\tilde E$. Later on, we will show that this ansatz is a minimiser
of the divergence. The key observation here is that the function
\begin{equation}\label{eq:bregman_curve}
x \mapsto \gamma(x) = x - \log\big(1+x\big),
\end{equation}

\begin{figure}
    \centering
    \includegraphics[scale=0.4]{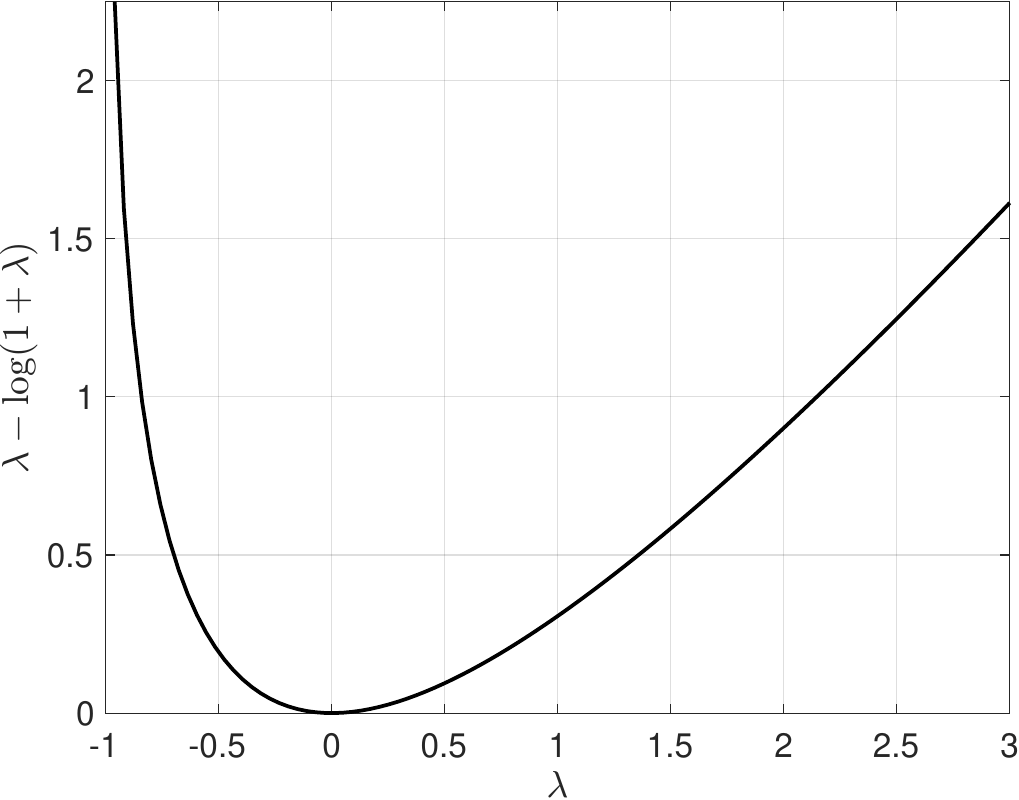}
    \caption{Bregman curve $\gamma$ in \cref{eq:bregman_curve}.}
    \label{fig:bregman_curve}
\end{figure}

is not an even function, as can be seen in \cref{fig:bregman_curve},
indeed $0.1931 = \gamma(- 0.5) > \gamma(0.5) = 0.0945$. In contrast, a TSVD
produces a rank $r$ approximation $\tsvd{\tilde E}{r}$ of the Hermitian matrix
$\tilde E$ by selecting the directions according to $x\mapsto |x|$, which by 
definition is an even function. A TSVD appraises both eigenvalues $\pm 0.5$
equally, whereas the Bregman log determinant divergence distinguishes between 
their contribution based on their sign. The following definition summarises 
this mathematically.

\begin{definition}[Bregman log determinant truncation]\label{def:Bregman_truncation}
Given $I + X \in \Snpp$ and the eigendecomposition $X=Z M Z\Hermitian\in\Sn$,
$M = \diag (\mu_1, \ldots, \mu_n)$, construct the index set
$\BregmanIndexSet{X}$ of cardinality $\# \BregmanIndexSet{X} = r<n$ by 
sorting the sequence
$\{\gamma(\mu_i)\}_{i=1}^n$ (the function $\gamma$ is defined in
\cref{eq:bregman_curve} shown in \cref{fig:bregman_curve}) and selecting
the indices corresponding to the largest values i.e.
\begin{equation}\label{eq:BregmanIndexSet}
\BregmanIndexSet{X} = \arg\max_{\substack{\mathcal{I} \subset \{1, \ldots, n\}\\ \#\mathcal{I} = r}} \sum_{j \in \mathcal{I}} \gamma(\mu_j).
\end{equation}
We define a BLD truncation of $X$ to order $r$ by
\begin{equation}\label{eq:Bregman_truncation}
\BregTrunc{X} = Z^\star M^\star (Z^\star)\Hermitian,
\end{equation}
where the columns $Z^\star \in \Cnr$ are constructed from $Z$ by deleting column
$j$ if $j\not\in\BregmanIndexSet{X}$. Similarly, the diagonal matrix $M^\star
\in \Crr$ is constructed from $M$ by deleting row $j$ and column $j$ if 
$j\not\in\BregmanIndexSet{X}$. 
\end{definition}

Before presenting the main result, we first look at an example using a diagonal matrix to develop some intuition.
\newcommand{\svdselect}[1]{\overbracket[1pt][3pt]{\strut #1 }^{\textbf{\color{cyan}{\emph{TSVD}}}}}
\newcommand{\bregselect}[1]{\underbracket[1pt][3pt]{\strut #1 }_{\textbf{\color{CarnationPink}{\emph{BLD}}}}}
\definecolor{darkgreen}{rgb}{0.21, 0.62, 0.22}
\newcommand{\svdselectgreen}[1]{\overbracket[1pt][3pt]{\strut #1 }^{\textbf{\color{darkgreen}{\emph{TSVD}}}}}
\newcommand{\bregselectgreen}[1]{\underbracket[1pt][3pt]{\strut #1 }_{\textbf{\color{darkgreen}{\emph{BLD}}}}}
\newcommand{\select}[1]{\svdselectgreen{\bregselectgreen{ #1 }}}

\begin{example}[Diagonal matrix]\label{ex:breg_vs_svd:diagonal}
Let us consider a concrete example of this approximation where $n=8$, $r=4$ and
$I + \tilde E$ is a diagonal matrix, where
\begin{align*}
\tilde E = \diag(\bregselect{-0.46}, \bregselect{-0.4}, -0.3, 0.18, \svdselect{0.5}, \svdselect{0.54}, \select{0.72}, \select{1.0}).
\end{align*}
\begin{figure}
    \centering
    \includegraphics[scale=0.45]{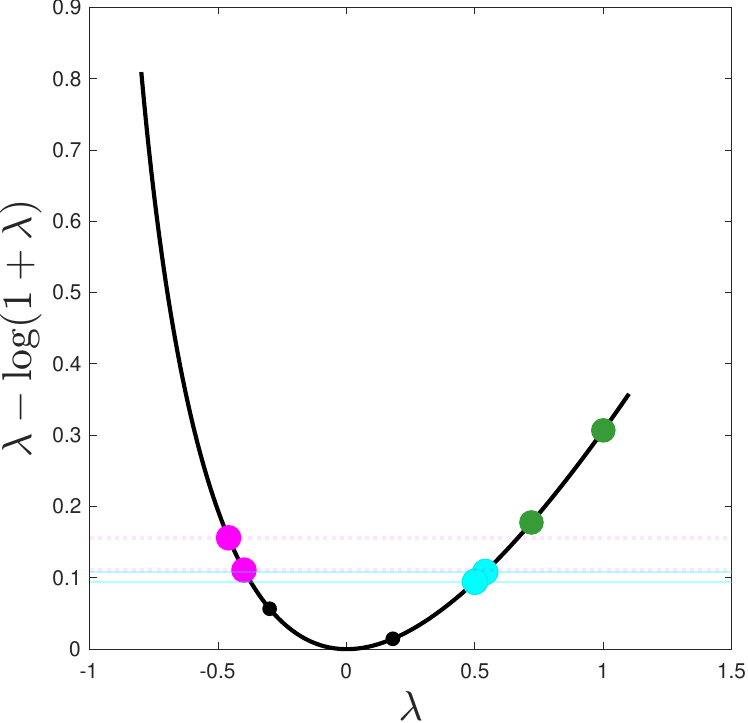}
    \includegraphics[scale=0.45]{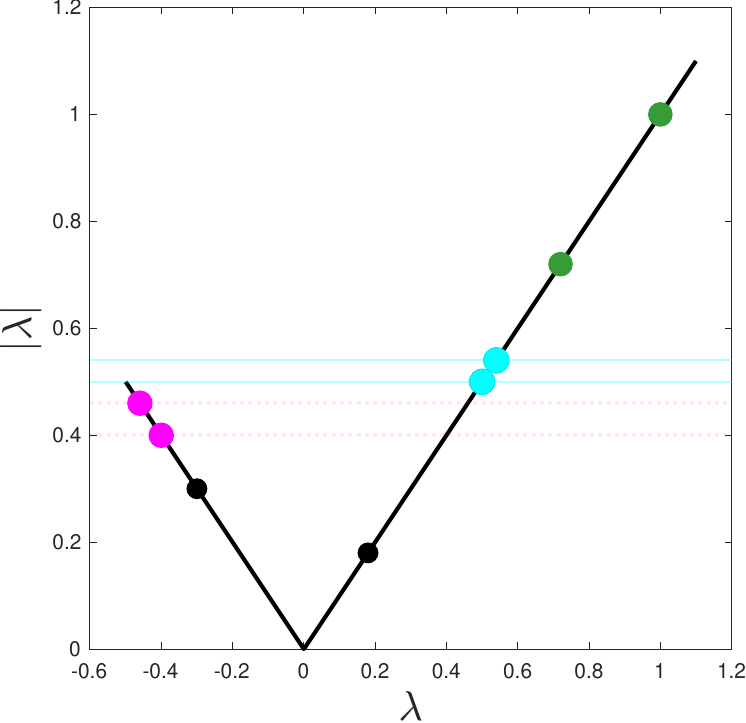}
    \caption{This depicts the situation in \cref{ex:breg_vs_svd:diagonal} where we 
    compare a TSVD and BLD truncations of a diagonal matrix $I + \tilde E$.
    The continuous curve traced on the left is \cref{eq:bregman_curve}.}
    \label{fig:Eigenvalues:U=V}
\end{figure}
The brackets indicate which eigenvalues are selected by the rank $r$ TSVD and
the BLD truncation to the same order. \Cref{fig:Eigenvalues:U=V} shows how the two
truncations differ. In the left figure, we show the image of these values under $\lambda\mapsto
\gamma(\lambda)$, whereas in the right we show them under that of $\lambda\mapsto |\lambda|$. 
It is clear that selecting the two negative eigenvalues leads to a smaller objective
value for $\gamma$ than those selected by a TSVD. The BLD truncation 
selects dominant directions according to the curve defined in \eqref{eq:bregman_curve}.
Furthermore,
\[
0.2381 = \BregmanLogDet(I + \tilde E, I + \BregTrunc{\tilde E}) < \BregmanLogDet(I + \tilde E, I + \tsvd{\tilde E}{r}) = 0.4764.
\]
\end{example}

We now prove the main result of this section.

\begin{theorem}\label{thm:Bregman_truncation:minimiser}
$\BregTrunc{\tilde E}$ is a minimiser of \cref{eq:Bregman_truncation:problem}.
\end{theorem}
\begin{proof}
We derive a lower bound for $W\mapsto \BregmanLogDet(I + \tilde E, I + W)$.
Recall from \cref{eq:G_and_W} that we have
\begin{align*}
& I + W = I + U \Lambda U^*,\\
& I + \tilde E = I + V \Theta V^*.
\end{align*}
Note that only $r$ values of $\lambda$ are non-zero. Using the bound from 
\cite{bushell1990trace} we have
\begin{align*}
\trace{(I + W)\inv (I + \tilde E)} \geq \sum_{i=1}^n \frac{1+\theta_i}{1+\lambda_i}.
\end{align*}
Note that
\begin{align*}
- \logdet{(I + W)\inv (I + \tilde E)} = - \sum_{i=1}^n \log\left(\frac{1+\theta_i}{1+\lambda_i}\right).
\end{align*}
It is easy to see that
\begin{align*}
\BregmanLogDet(I + \tilde E, I + W,) & \geq \sum_{i \not\in \{1,\ldots,n\}\setminus\BregmanIndexSet{\tilde E}} \Big(
\frac{1+\theta_i}{1+\lambda_i} - \log\left(\frac{1+\theta_i}{1+\lambda_i}\right)
- 1\Big)\\
& = \sum_{i \not\in \{1,\ldots,n\}\setminus\BregmanIndexSet{\tilde E}} \Big(
\theta_i - \log\left(1+\theta_i\right) \Big),
\end{align*}
and this lower bound is attained by the choice $W=\BregTrunc{\tilde E}$ in \cref{eq:P_choice}.
\end{proof}

\subsection{On the asymmetry of the Bregman divergence}\label{sec:asymmetry}

The Bregman log determinant divergence is not symmetric
with respect to its arguments characterised by the identity
\[
\BregmanLogDet(X, Y) = \text{trace}(XY^{-1} + YX^{-1}) - \BregmanLogDet(Y,X) - 2n.
\]
We now address this asymmetry in the context of deriving low-rank approximations.\\

In a previous work \cite{bock2023preconditioner}, the authors studied an 
optimisation problem similar to \cref{eq:Bregman_truncation:problem}, but
where the arguments of the Bregman divergence have been swapped:
\begin{subequations}\label{eq:Bregman_truncation:problem:inv}
\begin{align}
\min_{W\in\Sn} \quad & \BregmanLogDet(I + W, I + \tilde E)\\
\subto\quad
& I + W \in \Snpp\\
& \rank W \leq r.
\end{align}
\end{subequations}
It can be shown that this leads to a truncation of the matrix $\tilde E$ similar
to that given in \cref{def:Bregman_truncation}, but where $\gamma$ is replaced
by

\begin{figure}
    \centering
    \includegraphics[scale=0.4]{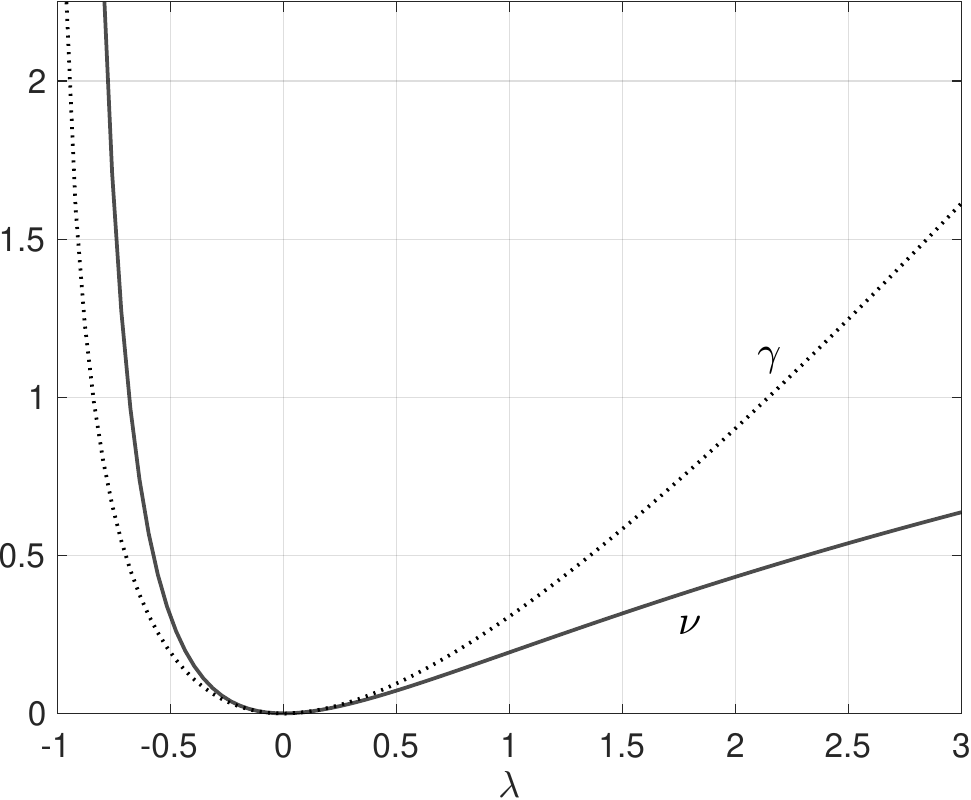}
    \caption{Reverse Bregman curve $\nu$ defined in \cref{eq:bregman_curve:inv}.
    The dashed line is the Bregman curve $\gamma$ from \cref{eq:bregman_curve}.}
    \label{fig:bregman_curve:inv}
\end{figure}

\begin{equation}\label{eq:bregman_curve:inv}
\nu(x) = \frac{1}{1 + x} - \log\Big(\frac{1}{1 + x}\Big) - 1,
\end{equation}
which is drawn in \cref{fig:bregman_curve:inv}.
\Cref{eq:bregman_curve:inv} leads to the following definition.
\begin{definition}[Reverse Bregman log determinant truncation]\label{def:Bregman_truncation:inverse}
Given $I + X \in \Snpp$ and the eigendecomposition $X=Z M Z\Hermitian\in\Sn$,
$M = \diag (\mu_1, \ldots, \mu_n)$, construct the index set
$\InverseBregmanIndexSet{X}$ of cardinality $\# \InverseBregmanIndexSet{X} = r<n$ by 
sorting the sequence
$\{\nu(\mu_i)\}_{i=1}^n$ and selecting the indices corresponding to the
largest values i.e.
\begin{equation}\label{eq:BregmanIndexSet:inverse}
\InverseBregmanIndexSet{X} = \arg\max_{\substack{\mathcal{I} \subset \{1, \ldots, n\}\\ \#\mathcal{I} = r}} \sum_{j \in \mathcal{I}} \nu(\mu_j).
\end{equation}
We define the \emph{reverse Bregman log determinant} (RBLD) truncation of $X$ to
order $r$ by
\begin{equation}\label{eq:Bregman_truncation:inverse}
\Sn\ni\BregTruncRev{X} = Z^\star M^\star (Z^\star)\Hermitian,
\end{equation}
where the columns $Z^\star \in \Cnr$ are constructed from $Z$ by deleting column
$j$ if $j\not\in\InverseBregmanIndexSet{X}$. Similarly, the diagonal matrix $M^\star
\in \Crr$ is constructed from $M$ by deleting row $j$ and column $j$ if 
$j\not\in\InverseBregmanIndexSet{X}$. 
\end{definition}

Finally, we present an important case of when the truncations coincide.

\begin{proposition}\label{cor:Br=Bbreg}
When $X\in\Snp$ with distinct non-zero eigenvalues, $\tsvd{X}{r} = \BregTrunc{X} = \BregTruncRev{X}$.
\end{proposition}
\begin{proof}
The eigenvalues of $I + \tilde X$ are bounded from below by $1$ when $X\in\Snp$ 
and since $\gamma$ (resp. $\nu$) is an increasing function on $[1, \infty)$
the index set $\BregmanIndexSet{X}$ (resp. $\InverseBregmanIndexSet{X}$)
constructed in \cref{eq:BregmanIndexSet} (resp. \cref{eq:BregmanIndexSet:inverse})
selects the same $r$ principal directions as a TSVD, so $\tsvd{X}{r} = \BregTrunc{X}$.
$\nu$ is an increasing function on $[1, \infty)$, so $\BregTrunc{X} = \BregTruncRev{X}$
since $X\in\Snp$.
\end{proof}
\section{A preconditioner with Bregman log determinant  truncation}\label{sec:preconditioner}

We propose the preconditioner
\begin{equation}\label{eq:Bregman_preconditioner}
\PrecondBreg{S} = Q(I + \BregTrunc{\tilde E})Q\Hermitian
\end{equation}
based on a BLD truncation of $\tilde E$. As we show next, this preconditioner 
is a solution to an optimisation problem.
\begin{theorem}\label{thm:Bregman_truncation:minimiser:preconditioner}
$\PrecondBreg{S}$ is a minimiser of
\begin{subequations}\label{eq:btSS}
\begin{align}
\minimise_{P\in\Snpp}\quad & \BregmanLogDet(S, P)\\
\subto\quad
& P = Q(I + W)Q\Hermitian,\label{eq:P_choice}\\
& W \in \Sn,\\
& \rank(W) = r.
\end{align}
\end{subequations}
\end{theorem}
\begin{proof}
Since the $\BregmanLogDet$ is invariant to congruence transformations, 
\[
\BregmanLogDet(S, P) = \BregmanLogDet(I + \tilde E, I + W),
\]
where $\tilde E = Q\inv E Q\invHermitian$. The result follows from 
\cref{thm:Bregman_truncation:minimiser}.
\end{proof}

\begin{remark}\label{remark:exact}
Obviously, the (reverse) Bregman truncation is exact if $\rank(\tilde E)
\leq r$. This highlights two interesting aspects of designing a preconditioner 
for $S$ within the framework presented here. Suppose we choose $Q$ as an
incomplete Cholesky factor of $S$. A fill-reducing algorithm can be used 
in such a procedure, in which case the factors are typically given up to
a permutation matrix $\mathcal{P}$
\[
Q = \mathcal{P}\inv C,
\]
which balances numerical accuracy against sparsity of $Q$. The permutation
matrix $\mathcal{P}$ will have an effect on $\tilde E$ (and its truncations).
Indeed, rank is not well-behaved under addition in the sense that for two 
Hermitian full rank matrices:
\[
0 \leq \rank(\tilde E) = \rank(S - QQ\Hermitian) \leq n.
\]
When applying the proposed framework, it may therefore also be of interest
to consider the interplay between $\mathcal{P}$ and the rank of $\tilde E$.
Roughly speaking, we therefore have three objectives: numerical
stability of $Q$, its structure or size (e.g. sparsity) and a reduction in
the rank of the term $S - QQ\Hermitian$.
\end{remark}

We define a preconditioner analogous to \cref{eq:Bregman_preconditioner}
but using the RBLD truncation in \cref{def:Bregman_truncation:inverse}.
\begin{equation}\label{eq:Bregman_preconditioner:reverse}
\PrecondBregRev{S} = Q(I + \BregTruncRev{\tilde E})Q\Hermitian.
\end{equation}

\begin{theorem}\label{thm:Bregman_truncation:minimiser:preconditioner:reverse}
$\PrecondBregRev{S}$ is a minimiser of
\begin{subequations}\label{eq:btSS:reverse}
\begin{align}
\minimise_{P\in\Snpp}\quad & \BregmanLogDet(P, S)\\
\subto\quad
& P = Q(I + W)Q\Hermitian,\\
& W \in \Sn,\\
& \rank(W) = r.
\end{align}
\end{subequations}

\end{theorem}
\begin{proof}
The proof is similar to that of \cref{thm:Bregman_truncation:minimiser:preconditioner}.
\end{proof}

We show empirically how these preconditioners perform in \cref{sec:numerical_results:ichol}
for modestly-sized matrices and develop some intuition for when they may differ from a
preconditioner based on a TSVD as in \cref{eq:scaled}. 
The preconditioners in \cref{eq:Bregman_preconditioner,eq:Bregman_preconditioner:reverse}
are expensive to construct in general, in particular for large matrices.
In \cref{sec:Bregman_directions}, we describe approximate truncations and analyse the
resulting preconditioners in \cref{sec:preconditioner:approx} (computational cost, storage,
application, etc.).
\section{Approximating the Bregman log determinant truncation}\label{sec:Bregman_directions}

The aim in this section is to present a heuristic to approximate
the BLD and RBLD truncations from \cref{sec:approximation_problem}.\\

We can regard $\BregTrunc{\tilde E}$ as consisting of two terms stemming
from the positive and negative semi-definite parts of $\tilde E$ that they
approximate. This can be viewed as splitting the "budget" $r$ into two
positive (and generally unknown) integers
\begin{equation}\label{eq:rsplit}
r = r^\oplus + r^\ominus,
\end{equation}
so that we can write a BLD truncation of $\tilde E$ as
\begin{align}
\BregTrunc{\tilde E}
& = \sum_{i=1}^{\rplus} \theta_i v_i v_i\Hermitian  + \sum_{i=n - \rminus + 1}^n \theta_i v_i v_i\Hermitian =: \BregTruncPositive{\tilde E} + \BregTruncNegative{\tilde E}.\label{eq:Bregtrunc_diagmatrix}
\end{align}
The main goal of this section is to describe how to approximate these 
matrices, which subsequently give rise to the preconditioner analysed
in \cref{sec:preconditioner:approx}.\\

As mentioned, the values of $\rplus$ and $\rminus$ in \cref{eq:rsplit} are not
known beforehand, so we introduce the parameter $\alpha\in [0,1]$ and, with a
slight abuse of notation, define the integers
\begin{subequations}\label{eq:rsplit:alpha}
\begin{align}
\rplusalpha & := \lfloor\alpha  r\rfloor,\\
\rminusalpha & := r - \rplusalpha,
\end{align}
\end{subequations}
representing the proportion of the rank $r$ budget to be allocated to
approximating the largest positive part of the spectrum of $\tilde E$. 
For a given $\alpha$, we define the approximations $\BregmanApproxAlphaPositive$ and
$\BregmanApproxAlphaNegative$
\begin{align*}
\BregmanApproxAlphaPositive & := \sum_{i=1}^{\rplusalpha} \theta_i v_i v_i\Hermitian,\\
\BregmanApproxAlphaNegative & := \sum_{i=n - \rminusalpha + 1}^n \theta_i v_i v_i\Hermitian,
\end{align*}
and set
\begin{equation}\label{eq:BregmanApproxAlpha}
\BregmanApproxAlpha := \BregmanApproxAlphaPositive + \BregmanApproxAlphaNegative.
\end{equation}
In general, the values of $\alpha$ for which $\BregmanApproxAlpha
= \BregTrunc{\tilde E}$ or $\BregmanApproxAlpha = \tsvd{\tilde E}{r}$ are not
known. We return to this in \cref{sec:numerical_results:large}.
What remains of this section is dedicated to describing how the two terms in
the approximation \cref{eq:BregmanApproxAlpha} can be computed in practice.
The construction of $\BregmanApproxAlphaPositive$ and $\BregmanApproxAlphaNegative$ 
will be discussed in \cref{sec:Bregman_directions:largest} and
\cref{sec:Bregman_directions:smallest}, respectively.

\subsection{Estimating the largest eigenvalues}\label{sec:Bregman_directions:largest}

In this section, we describe how to compute approximations of the matrix
$\BregmanApproxAlphaPositive$ in \cref{eq:BregmanApproxAlpha} which represent
the $\rplusalpha$ largest positive eigenvalues of $\tilde E$. Note that, by definition,
\begin{equation}\label{eq:tildeE}
\tilde E = Q\inv SQ\invHermitian - I,
\end{equation}
so once $\alpha$ is prescribed, we can compute only with the HPD matrix 
$Q\inv SQ\invHermitian$ in practice.
Instead, we estimate the largest eigenvalues of $Q\inv SQ\invHermitian$ using
a Krylov--Schur method \cite{stewart2002krylov}, which only requires 
matrix-vector products with $Q\inv$, $Q\invHermitian$ and $S$. This
method is also used in \MATLAB 2023b's implementation of \texttt{eigs}.
Since $S$ is symmetric, this is a special case of the \emph{thick-restart
Lanczos method} of \cite{wu2000thick}. This method provides a matrix
$H^\oplus \in \mathbb{C}^{n\times\rplusalpha}$ and a diagonal matrix
$D^\oplus \in \mathbb{R}^{\rplusalpha\times\rplusalpha}$ approximating
the largest eigenvalues of $Q\inv SQ\invHermitian$.\\

In view of \cref{eq:tildeE}, we therefore use the following approximation in
\cref{sec:numerical_results:large}:
\begin{equation}\label{eq:BregmanApproxAlphaPositive}
\BregmanApproxAlphaPositive := H^\oplus \big(D^\oplus - I\big) (H^\oplus)\Hermitian.
\end{equation}

\subsection{Smallest eigenvalues}\label{sec:Bregman_directions:smallest}
\newcommand{\lambdamax}{\eta}

Approximating $\BregmanApproxAlphaNegative$ in 
\cref{eq:BregmanApproxAlpha} is tantamount to approximating the eigenspaces
of $Q\inv SQ\invHermitian$ associated with its smallest eigenvalues.
These can be estimated using a \emph{shift-and-invert}
power method \cite{golub2013matrix} which, for some $\lambdamax>0$ requires 
the solution of linear systems of equations involving the matrix $\lambdamax
I - Q\inv SQ\invHermitian$. This can be as expensive as solving the original
problem \cref{eq:Sxb} to begin with. Provided an estimate $\lambdamax$ of
the largest eigenvalue $\theta_1$ of $Q\inv SQ\invHermitian$ such that
$\lambdamax\geq\theta_1$, the matrix
\begin{equation}\label{eq:tildeE:shifted}
\lambdamax I - Q\inv SQ\invHermitian
\end{equation}
is positive semi-definite. The eigenvalues $\theta_1 \geq \ldots \geq
\theta_n$ of $\tilde E$ satisfy
\begin{equation}\label{eq:shift}
\theta_{n-i+1} = \lambdamax - \lambda_i(\lambdamax I - Q\inv SQ\invHermitian) - 1, \quad i = 1,\ldots n,
\end{equation}
where $\lambda_i(\lambdamax I - Q\inv SQ\invHermitian)$ denotes the $i^\textnormal{th}$
largest eigenvalue of $\lambdamax I - Q\inv SQ\invHermitian$.
We can therefore estimate the smallest negative eigenvalues of $\tilde E$
by computing the largest eigenvalues of the shifted matrix
\cref{eq:tildeE:shifted}, which obviates the need for a shift-and-invert
approach. 


The aim is to solve the eigenvalue
problem
\[
\big(\lambdamax I - Q\inv S Q\Hermitian\big) x = \lambda x,
\]
for the $\rminusalpha$ largest eigenvalues and eigenvectors. We can then
compute $\BregmanApproxAlphaNegative$ via \cref{eq:shift}. The Krylov--Schur
method depends on several parameters; the maxmimum size of the subspace,
number of iterations and a tolerance at which eigenvalues are deemed to have
converged. We comment on these in \cref{sec:numerical_results:large}.\\

The strategy for computing $\BregmanApproxAlphaNegative$ is summarised below.

\begin{enumerate}
    \item Estimate $\lambdamax \approx \theta_1$, which may already be available from $\BregmanApproxAlphaPositive$ (or can be computed using the steps in \cref{sec:Bregman_directions:largest}).
    \item Compute $H^\ominus\in\mathbb{C}^{n\times\rminusalpha}$ with orthonormal columns and a diagonal matrix $D^\ominus \in\mathbb{R}^{\rminusalpha\times\rminusalpha}$ representing the largest eigenvalues of $\lambdamax I - Q\inv S Q\Hermitian$ using the Krylov--Schur method mentioned above. This means that $H$ and $D$ approximately satisfy
        \[
            H^\ominus D^\ominus(H^\ominus)\Hermitian \approx \tsvd{\lambdamax I - \tilde E}{r}.
        \]
    \item Undo the shift in \cref{eq:shift} and set
        \[
            \BregmanApproxAlphaNegative := H^\ominus\Big((\lambdamax - 1) I - D^\ominus\Big) (H^\ominus)\Hermitian.
        \]
\end{enumerate}

In principle, a Nyström approximation \cite{martinsson2020randomized} could also be used in step 2, but we have
found that it does not result in good approximations of $\tsvd{\lambdamax I - \tilde E}{r}$.
We believe this is because the gap between its largest eigenvalues (resp. the smallest
eigenvalues of $\tilde E$) is small owing to \cref{eq:eigenvalue_interval} and
therefore makes it difficult to compute them accurately.
The choice of method used in step 2 depends on the specific matrix $S$ and the
level of fill in the approximate factor $Q$. The same applies to the 
approximation in \cref{sec:Bregman_directions:largest}, which enters the approximation
of $\BregmanApproxAlphaNegative$ through $\lambdamax$.

\section{Approximating the proposed preconditioner}\label{sec:preconditioner:approx}

\Cref{eq:BregmanApproxAlpha} introduced $\BregmanApproxAlpha = \BregmanApproxAlphaPositive
+ \BregmanApproxAlphaNegative$ as a way of approximating $\BregTrunc{\tilde E}$ as a function
of $\alpha \in [0, 1]$. An approximation of the proposed preconditioner
in \cref{eq:Bregman_preconditioner} is therefore given by
\begin{equation}\label{eq:Bregman_preconditioner:approx}
\PrecondBregAlpha{S}{\alpha} = Q(I + \BregmanApproxAlpha)Q\Hermitian.
\end{equation}

In practice, the matrix in \cref{eq:Bregman_preconditioner:approx} is not
formed explicitly, and only the applications of $Q\inv$, its Hermitian
adjoint, and $(I + \BregmanApproxAlpha)\inv$ are necessary. Given a spectral
decomposition of $\BregmanApproxAlpha$, the cost of applying the latter is 
$O(nr)$ by the Woodbury matrix identity. As a result, applying the preconditioner
to a vector costs
\[
O(\MatVec_Q + nr).
\]

Cost of storage and construction are linked through the choice of $Q$. A sparse
approximation $QQ\Hermitian$ to $S$ may reduce storage requirements and time
to apply the inverse of $Q$ (and its Hermitian adjoint), possibly at the expense
of a larger error $\tilde E$. For instance, if $Q$ were to be computed as an 
incomplete Cholesky with a non-zero drop threshold (i.e. entries of the Cholesky
factors of $S$ above this threshold are included in $Q$), then we would expect 
a better approximate factorisation for smaller values of such a tolerance at
the cost of potentially denser factors. We therefore expect to pay for better
approximations $QQ\Hermitian \approx S$ by increasing storage cost through $Q$,
which possibly reduces both the magnitude of the extremal eigenvalues of $\tilde E$
and its rank. Irrespective of $\alpha$, the preconditioner requires storage of 
$O(\texttt{nnz}(Q) + nr)$.\\

\begin{figure}
    \centering
    \includegraphics[width=0.48\linewidth]{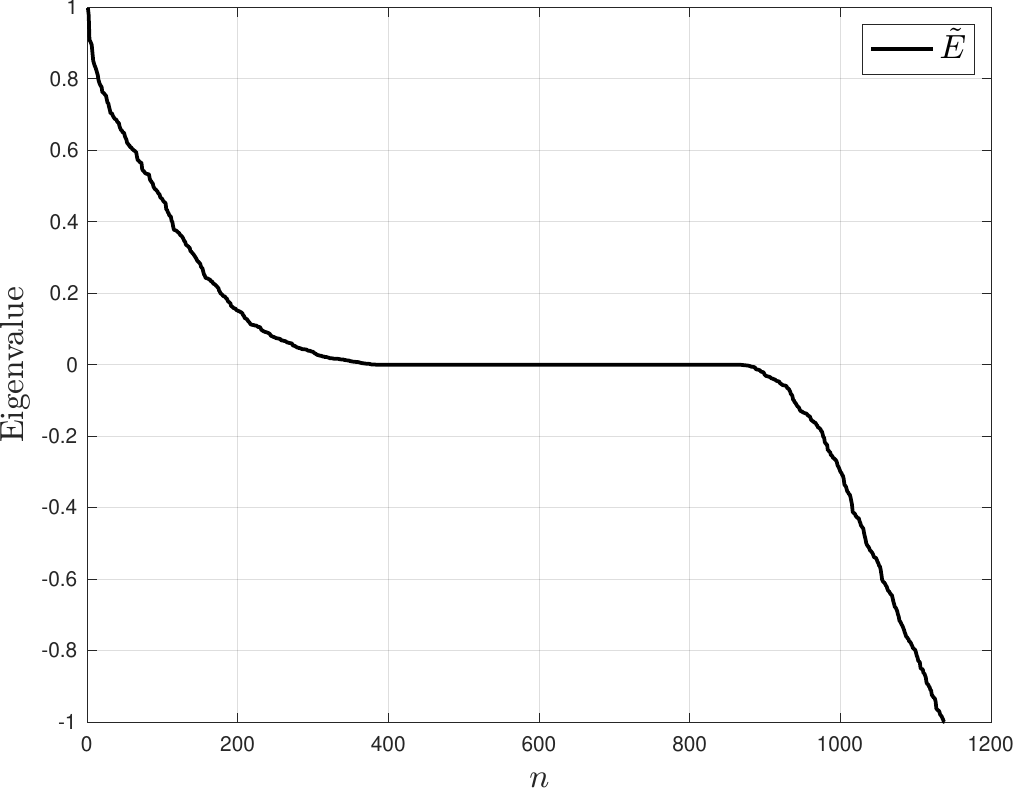}
    \includegraphics[width=0.47\linewidth]{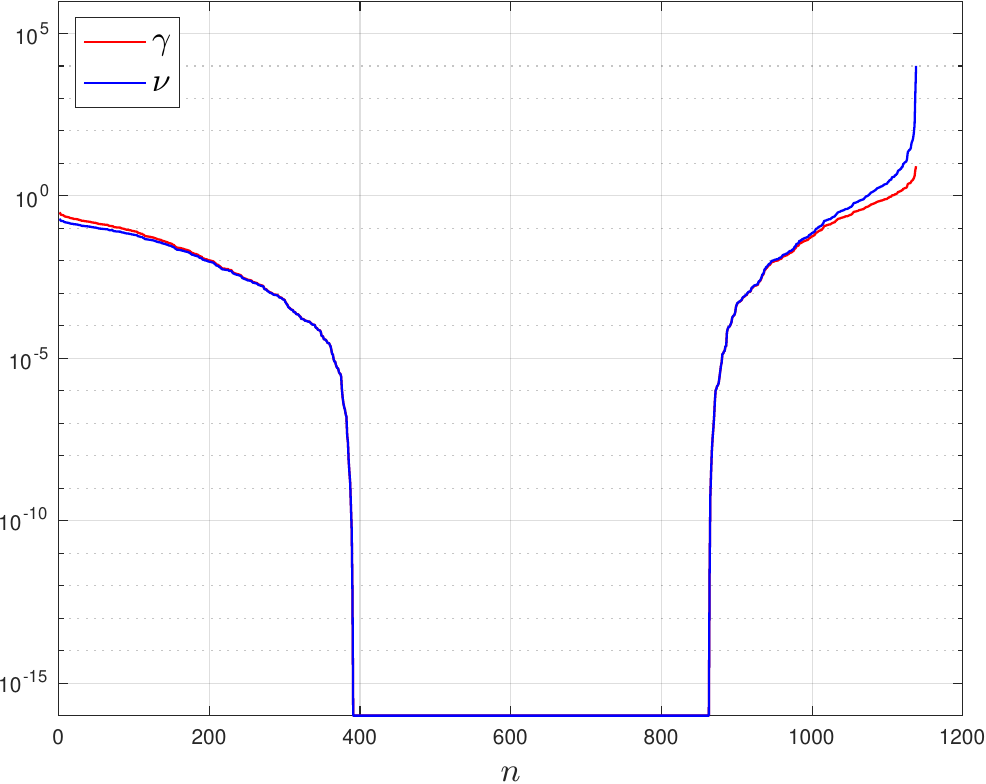}
    \caption{Left: spectrum of $\tilde E = Q\inv S Q\invHermitian - I$, where $S$ 
    is the matrix \texttt{HB/1138\_bus} taken from SuiteSparse, and $Q$ is its zero fill
    incomplete Cholesky factorisation. Right: values of $\tilde E$ under the image of $\gamma$ and $\nu$ (defined in \cref{eq:bregman_curve} and \cref{eq:bregman_curve:inv}, respectively).}
    \label{fig:spectrum_tildeE}
\end{figure}

\Cref{fig:spectrum_tildeE} shows the spectrum of $\tilde E$ for one
of the examples in \cref{sec:numerical_results:ichol}.
If the numerical rank of $\tilde E$  is large, we may also require a large
$r$ in the preconditioner \cref{eq:Bregman_preconditioner:approx}.
Even if the numerical rank of $\tilde E$ is small, there is no guarantee
that the largest positive and smallest negative eigenvalues are small in magnitude.
Roughly speaking, a steep slope in the spectrum can make it easier for Krylov--Schur
to converge, as the largest eigenvalues are more isolated from the rest of the spectrum.
Knowledge of the spectrum of $\tilde E$ can therefore inform the choice of the
method used to construct $\BregmanApproxAlphaPositive$.
$\alpha$ also affects construction costs through \cref{eq:rsplit:alpha}.
We expect that it is easier and cheaper to compute $\BregmanApproxAlphaPositive$ than
$\BregmanApproxAlphaNegative$, but it depends on the spectrum, as mentioned above.
Its construction costs $O(r^3)$ operations due to the inverse needed to apply $(I
+ \BregmanApproxAlpha)\inv$ in excess of costs 
discussed in \cref{sec:Bregman_directions:largest,sec:Bregman_directions:smallest}.\\

In summary, if $S$ is large and sparse, a good candidate $Q$ for use in 
\cref{eq:Bregman_preconditioner:approx} keeps both the numerical rank of
$\tilde E$ small and reduces the condition number of $Q\inv SQ\invHermitian$.
From what precedes, we know that the Bregman log determinant divergence takes on
large values for small negative eigenvalues of $\tilde E$ close to $-1$ (cf. 
\cref{fig:bregman_curve,eq:Divergence:U=V}). As a
result, an approximate factorisation should try to eliminate these from the
error so that the $\tilde E$ is "as close to" positive semi-definite as possible.
In this case, we will expect the Nyström approximation to perform well and be
a good alternative to the Krylov--Schur method mentioned above. We describe this
method in more detail in the next section.
The authors are not aware of methods of generating $Q$ that can control the 
inertia of the $S - QQ\Hermitian$, although this would of course
be of interest.

\subsection{Randomised methods}

Another option to approximate $\BregTrunc{\tilde E}$ is to adopt a randomised
approach. Randomised linear algebra \cite{halko2011finding} encompasses a popular
family of algorithms that can yield good low-rank approximations. These are 
particularly useful when it is prohibitively expensive to compute a truncated
SVD directly, since they only require the action of a matrix. The \emph{randomised
SVD} \cite[Algorithm 8]{martinsson2020randomized} (RSVD) is applicable to general
square matrices $X$ of order $n$. To construct a RSVD, $X$ is first \emph{sketched}
by multiplication with an $n\times r$ matrix $\Omega$ with e.g. normal Gaussian 
i.i.d. entries:
\begin{equation}\label{eq:sketch}
Y = X\Omega.
\end{equation}
A reduced $QR$ decomposition of $Y$ yields a $n\times r$ matrix $\mathcal{O}$
with orthonormal columns whose range approximates that of $X$. Using
\MATLAB-style notation, we can compute a reduced SVD
\[
\mathbf{\hat{U}} \mathbf{\Sigma} \mathbf{V}\Hermitian = \texttt{svd\_econ}(X,\,r).
\]
Setting $\mathbf{U} :=\mathcal{O} \mathbf{\hat{U}}$ we obtain the approximation
\[
X \approx \mathbf{U} \mathbf{\Sigma} \mathbf{V}\Hermitian.
\]
We denote by $\MatVec_X$ the cost of computing matrix-vector products with $X$.
The randomised SVD comes at a cost of $O(\MatVec_X r + nr^2)$ which is an
improvement on the naïve $O(n^3)$ cost of truncating a full SVD.\\

For Hermitian matrices, one of the most popular methods in randomised linear 
algebra is the \emph{Nystr\"om} approximation
\cite{gittens2016revisiting,drineas2005nystrom,nystrom1930praktische}. We
define the rank $r$ Nyström approximation of $X\in\Sn$ with respect to a 
Gaussian sketching matrix $\Omega \in \Rnk$, $k\geq r$:
\cite{tropp2017fixed}
\begin{equation}\label{eq:Nystrom}
\Nystrom{X} := X\Omega \tsvd{\Omega\Hermitian X \Omega}{r}^+ (X\Omega)\Hermitian.
\end{equation} 
The cost of this approximation is $O(\MatVec_X k + k^3)$
in construction cost and requires only $k$
products with the matrix. This approximation can be derived as range-restricted
matrix approximation that is optimal in the Bregman divergence
\cite{bock2023preconditioner} which corresponds to a single iteration of
Bregman's algorithm for rank-deficient matrices with no constraints 
\cite[Section 4.3]{kulis2009low}. The approaches above can be combined with
oversampling and power iterations, and can also be adapted to use only a single
"view" of the matrix via \cref{eq:sketch}. See
\cite{martinsson2020randomized,halko2011finding} for comprehensive reviews.\\

Modifications can be made to the Nystr\"om approximation in order 
to produce a numerically robust procedure for indefinite Hermitian matrices 
\cite{nakatsukasa2023randomized}. It is similar
to \cref{eq:Nystrom}, where the sketching matrix $\Omega$ is a $n\times cr$ 
matrix for some appropriate oversampling $c>1$ (e.g. $c=1.5$, $c=2$). We
repeat the definition for convenience:
\begin{equation}\label{eq:Nystrom:indefinite}
\NystromIndefinite{X} := X\Omega \tsvd{\Omega\Hermitian X \Omega}{r}^+ (X\Omega)\Hermitian.
\end{equation} 
By more generous oversampling and truncating in the inner
matrix, $\tsvd{\Omega\Hermitian X \Omega}{r}$ we avoid taking the inverse of very
small eigenvalues, which could blow up and render the approximation 
$\NystromIndefinite{X}$ inaccurate. In \cref{sec:numerical_results:large},
a comparison will be made between a preconditioner based on this approximation 
along with a numerical study of how different choices of $\alpha$ in 
\cref{eq:Bregman_preconditioner:approx} affects the performance of PCG.

\section{Numerical experiments with PCG}\label{sec:numerical_results}

In this section, we investigate the performance of the preconditioners numerically.
In \cref{sec:numerical_results:ichol}, we consider small matrices for
which the low-rank approximations associated with the preconditioners $\PrecondBreg{S}$,
$\PrecondBregRev{S}$ and $\PrecondSVD{S}$ can be computed to machine precision. In
\cref{sec:numerical_results:large}, we use larger instances where we use the
preconditioner $\PrecondBregAlpha{S}{\alpha}$ and other practical approximations.\\

The numerical results shown below were generated using \MATLAB 2023b.
The code to reproduce them is available at \url{https://www.github.com/andreasbock/bldp-matlab}
and was run on a node featuring two Intel Xeon Processor 2650v4 (12 core, 2.20GHz) with
256 GB memory (16 x 16 GB DDR4-2666).

\subsection{Improving incomplete Cholesky using low-rank compensation}\label{sec:numerical_results:ichol}

In this section, we show empirically that PCG using a preconditioner based on a BLD 
truncation can reduce the number of iterations when compared to an analogous preconditioner
based on a TSVD. We select matrices from the SuiteSparse Matrix Collection\footnote{Available
at \url{https://sparse.tamu.edu/}.} \cite{kolodziej2019suitesparse}, as shown in 
\cref{table:ichol_problems}. We consider matrices of modest size in order to compute
BLD truncations and TSVDs with high accuracy.
Some of these correspond to least-squares problems involving $A\in\Rmn$, for which we
consider the corresponding normal equations $S = A\Hermitian b$, for $S=A\Hermitian A$
and a random non-zero $b\in\Realn$. If the matrix is symmetric positive definite,
we simply consider $Sx=b$.\\

\begin{table}
\caption{Test set from the SuiteSparse Matrix Collection. We consider both
symmetric positive definite (SPD) matrices of size $n\times n$ and $m \times n$ matrices
for $m>n$ (transposing if necessary). For the latter, we solve the associated normal 
equations. The number of non-zero entries in the matrix is denoted by $\texttt{nnz}$.
}\label{table:ichol_problems}
\begin{center}
\begin{tabular}{l r r r}
\toprule 
Problem (SPD) & & $n$ & $\texttt{nnz}$\\\midrule
\texttt{HB/bcsstk22} &  & 138 & 696\\
\texttt{HB/lund\_a}  & & 147 & 2449\\
\texttt{HB/bcsstk05}  & & 153 & 2423\\
\texttt{Pothen/mesh2e1}  & & 306 & 2018\\
\texttt{HB/494\_bus}  & & 494 & 1666\\
\texttt{HB/662\_bus}  & & 662 & 2474\\
\texttt{HB/bcsstk08}  & & 1074 & 12960\\
\texttt{HB/1138\_bus} &  & 1138 & 4054\\\midrule
Problem (normal equations) & $m$ & $n$ & $\texttt{nnz}$\\\midrule
\texttt{Meszaros/l9} & 1483 & 244 & 4659\\
\texttt{LPnetlib/lp\_sctap1} & 660 & 300 & 1872\\
\texttt{LPnetlib/lp\_bandm} & 472 & 305 & 2494\\
\texttt{HB/illc1850} & 1850 & 712 & 8636\\
\texttt{LPnetlib/lp\_sctap3} & 3340 & 1480 & 9734\\\bottomrule
\end{tabular}
\end{center}
\end{table}

We let $QQ\Hermitian \approx S$ be given as a zero fill incomplete 
Cholesky factorisation of a target matrix $S$, and consider the following
cases:
\begin{enumerate}
    \item no preconditioner,
    \item symmetric preconditioning using the zero fill incomplete Cholesky factors $Q$ described above,
    \item left-preconditioning using $\PrecondSVD{S} = Q(I + \tsvd{E}{r})Q\Hermitian$ (\cref{eq:scaled}),
    \item left-preconditioning using $\PrecondBreg{S} = Q(I + \BregTrunc{E})Q\Hermitian$ (\cref{eq:Bregman_preconditioner}),
    \item and left-preconditioning using $\PrecondBregRev{S} = Q(I + \BregTruncRev{E})Q\Hermitian$ (\cref{eq:Bregman_preconditioner:reverse}).
\end{enumerate}

A major disadvantage of the proposed preconditioners is that it depends heavily on
the spectrum of $\tilde E$. An appropriate choice of rank $r$ that balances the
computational cost of the truncations with any gain related to convergence of PCG is
typically not known beforehand. Here, the focus is on experimenting with
the preconditioners above and study how the truncations differ.
We therefore use values of $r$ relative to $n$ for the low-rank approximations
\begin{equation}\label{eq:rank_variation}
r = \lfloor n \times \epsilon \rfloor, \quad \epsilon \in \{ 0.01, 0.05, 0.1\}.
\end{equation}

\edef\totalrows{\thecsvrow}%
\begin{table}[tbhp]
\scriptsize
\caption{Results for the matrices in \cref{table:ichol_problems} selected 
from the SuiteSparse Matrix Collection.
The residual tolerance was set to $10^{-10}$. For each row, we have highlighted the
smallest number across the relevant columns. The $\dagger$ symbol on the iterations
means that the associated low-rank truncations coincide.}\label{table:ichol_results}
\begin{center}
\begin{tabular*}{\textwidth}{l c @{\extracolsep{\fill}} *{13}{c} }
\toprule 
\multirow{2.2}{*}{Problem} & \multirow{2.2}{*}{$n$} & \multirow{2.2}{*}{$r$} & \multicolumn{5}{c}{Iteration count} & \multicolumn{3}{c}{$\kappa_2$} & \multicolumn{3}{c}{$\BregmanLogDet(S, \cdot)$}\\ \cmidrule(l{0pt}r{2pt}){4-8} \cmidrule(l{2pt}r{2pt}){9-11} \cmidrule(l{2pt}r{2pt}){12-14}
\multicolumn{3}{c}{} & $I$ & $Q$ & $\PrecondBregRev{S}$ & $\PrecondSVD{S}$ & $\PrecondBreg{S}$ & $\PrecondBregRev{S}$ & $\PrecondSVD{S}$ & $\PrecondBreg{S}$ & $\PrecondBregRev{S}$ & $\PrecondSVD{S}$ & $\PrecondBreg{S}$\\
\midrule
\begin{filecontents*}{csvs/small/results_out.csv}
Name,n,r,resnopc,resichol,ressvd,resbreg,resrbreg,iternopc,iterichol,itersvd,iterbreg,iterrbreg,condnopc,condichol,condsvd,condbreg,condrbreg,divnopc,divichol,divsvd,divbreg,divrbreg,flagnopc,flagichol,flagsvd,flagbreg,flagrbreg,switch,BeqR,BeqS,ReqS
HB/bcsstk22,138,2,2.0e+00,6.1e-11,6.4e-11,4.1e-11,4.1e-11,-,41,37,\textbf{32}$^\dagger$,\textbf{32}$^\dagger$,1.7e+05,3.1e+03,1.2e+03,\textbf{2.3e+02},\textbf{2.3e+02},-inf,1.9e+01,1.6e+01,\textbf{1.1e+01},\textbf{1.1e+01},1,0,0,0,0,0,1,0,0
HB/bcsstk22,138,6,2.0e+00,6.1e-11,2.3e-11,2.3e-11,2.9e-11,-,41,24$^\dagger$,24$^\dagger$,\textbf{23},1.7e+05,3.1e+03,\textbf{2.1e+01},\textbf{2.1e+01},2.2e+01,-inf,1.9e+01,\textbf{4.2e+00},\textbf{4.2e+00},6.0e+00,1,0,0,0,0,0,0,1,0
HB/bcsstk22,138,13,2.0e+00,6.1e-11,3.2e-11,8.4e-11,5.4e-11,-,41,19,\textbf{17},\textbf{17},1.7e+05,3.1e+03,7.8e+00,7.2e+00,\textbf{5.6e+00},-inf,1.9e+01,1.9e+00,\textbf{1.8e+00},1.9e+00,1,0,0,0,0,1,0,0,0
HB/lund_a,147,2,9.7e-01,5.2e-11,2.8e-11,2.8e-11,8.9e-11,-,20,16$^\dagger$,16$^\dagger$,\textbf{15},5.4e+06,6.6e+02,\textbf{1.3e+01},\textbf{1.3e+01},2.1e+01,-inf,4.6e+00,\textbf{1.2e+00},\textbf{1.2e+00},1.3e+00,1,0,0,0,0,0,0,1,0
HB/lund_a,147,7,9.7e-01,5.2e-11,7.9e-11,4.0e-11,4.0e-11,-,20,\textbf{12},\textbf{12}$^\dagger$,\textbf{12}$^\dagger$,5.4e+06,6.6e+02,3.4e+00,3.1e+00,\textbf{3.1e+00},-inf,4.6e+00,3.2e-01,\textbf{3.1e-01},\textbf{3.1e-01},1,0,0,0,0,0,1,0,0
HB/lund_a,147,14,9.7e-01,5.2e-11,3.8e-11,2.6e-11,2.6e-11,-,20,\textbf{10},\textbf{10}$^\dagger$,\textbf{10}$^\dagger$,5.4e+06,6.6e+02,2.3e+00,2.2e+00,\textbf{2.2e+00},-inf,4.6e+00,1.7e-01,\textbf{1.7e-01},\textbf{1.7e-01},1,0,0,0,0,1,1,0,0
HB/bcsstk05,153,2,2.0e-01,8.6e-11,6.7e-11,6.8e-11,6.8e-11,-,40,35,\textbf{32}$^\dagger$,\textbf{32}$^\dagger$,3.5e+04,2.3e+03,1.5e+03,\textbf{1.4e+02},1.6e+02,-inf,1.8e+01,1.3e+01,\textbf{1.0e+01},\textbf{1.0e+01},1,0,0,0,0,0,1,0,0
HB/bcsstk05,153,7,2.0e-01,8.6e-11,2.4e-11,3.7e-11,6.7e-11,-,40,24,22,\textbf{21},3.5e+04,2.3e+03,7.9e+01,2.2e+01,\textbf{2.0e+01},-inf,1.8e+01,4.7e+00,\textbf{3.9e+00},3.9e+00,1,0,0,0,0,0,0,0,0
HB/bcsstk05,153,15,2.0e-01,8.6e-11,2.7e-11,4.9e-11,8.6e-11,-,40,18,17,\textbf{16},3.5e+04,2.3e+03,1.2e+01,1.1e+01,\textbf{8.9e+00},-inf,1.8e+01,2.3e+00,\textbf{2.1e+00},2.2e+00,1,0,0,0,0,1,0,0,0
Meszaros/l9,244,2,6.1e-10,4.2e-11,8.9e-11,6.8e-11,7.2e-11,-,34,33,29,\textbf{27},8.0e+04,9.5e+01,1.3e+02,5.2e+01,\textbf{4.4e+01},3.0e+03,2.0e+01,1.7e+01,\textbf{1.7e+01},1.8e+01,1,0,0,0,0,0,0,0,0
Meszaros/l9,244,12,6.1e-10,4.2e-11,9.6e-11,2.1e-11,7.6e-11,-,34,26,22,\textbf{20},8.0e+04,9.5e+01,8.3e+01,2.4e+01,\textbf{2.2e+01},3.0e+03,2.0e+01,9.4e+00,\textbf{7.7e+00},9.9e+00,1,0,0,0,0,0,0,0,0
Meszaros/l9,244,24,6.1e-10,4.2e-11,1.4e-11,5.7e-11,2.3e-11,-,34,\textbf{15},\textbf{15},\textbf{15},8.0e+04,9.5e+01,1.2e+01,\textbf{9.4e+00},1.1e+01,3.0e+03,2.0e+01,3.4e+00,\textbf{3.1e+00},3.4e+00,1,0,0,0,0,1,0,0,0
LPnetlib/lp_sctap1,300,3,8.2e-04,1.8e-07,6.6e-07,2.8e-09,2.8e-09,-,-,-,-$^\dagger$,-$^\dagger$,8.6e+04,9.1e+03,6.7e+03,\textbf{4.0e+03},\textbf{4.0e+03},-inf,1.5e+02,1.5e+02,\textbf{1.4e+02},\textbf{1.4e+02},1,1,1,1,1,0,1,0,0
LPnetlib/lp_sctap1,300,15,8.2e-04,1.8e-07,1.1e-09,3.9e-11,3.9e-11,-,-,-,\textbf{81}$^\dagger$,\textbf{81}$^\dagger$,8.6e+04,9.1e+03,2.9e+03,\textbf{1.1e+03},\textbf{1.1e+03},-inf,1.5e+02,1.3e+02,\textbf{8.4e+01},\textbf{8.4e+01},1,1,1,0,0,0,1,0,0
LPnetlib/lp_sctap1,300,30,8.2e-04,1.8e-07,4.2e-11,7.2e-11,7.2e-11,-,-,71,\textbf{50}$^\dagger$,\textbf{50}$^\dagger$,8.6e+04,9.1e+03,8.3e+02,5.2e+02,\textbf{5.1e+02},-inf,1.5e+02,7.8e+01,\textbf{3.4e+01},\textbf{3.4e+01},1,1,0,0,0,1,1,0,0
LPnetlib/lp_bandm,305,3,4.4e-04,9.0e-11,5.4e-11,8.0e-11,6.4e-11,-,55,50,40,\textbf{39},3.1e+07,3.2e+03,1.2e+03,\textbf{3.1e+02},4.1e+02,1.3e+06,3.1e+01,2.6e+01,\textbf{2.2e+01},2.3e+01,1,0,0,0,0,0,0,0,0
LPnetlib/lp_bandm,305,15,4.4e-04,9.0e-11,3.1e-11,2.4e-11,6.8e-11,-,55,27,25,\textbf{24},3.1e+07,3.2e+03,7.0e+01,7.2e+01,\textbf{5.4e+01},1.3e+06,3.1e+01,1.2e+01,\textbf{1.1e+01},1.1e+01,1,0,0,0,0,0,0,0,0
LPnetlib/lp_bandm,305,30,4.4e-04,9.0e-11,1.0e-10,4.7e-11,5.6e-11,-,55,\textbf{17},\textbf{17},\textbf{17},3.1e+07,3.2e+03,1.9e+01,2.1e+01,\textbf{1.6e+01},1.3e+06,3.1e+01,5.9e+00,\textbf{5.5e+00},5.7e+00,1,0,0,0,0,1,0,0,0
Pothen/mesh2e1,306,3,2.3e-08,1.1e-11,9.4e-11,1.6e-11,1.6e-11,-,23,19,\textbf{18}$^\dagger$,\textbf{18}$^\dagger$,4.2e+02,1.5e+01,9.4e+00,\textbf{4.4e+00},\textbf{4.4e+00},7.3e+03,4.1e+00,2.6e+00,\textbf{1.9e+00},\textbf{1.9e+00},1,0,0,0,0,0,1,0,0
Pothen/mesh2e1,306,15,2.3e-08,1.1e-11,2.8e-11,3.2e-11,3.2e-11,-,23,\textbf{11},\textbf{11}$^\dagger$,\textbf{11}$^\dagger$,4.2e+02,1.5e+01,2.0e+00,\textbf{2.0e+00},\textbf{2.0e+00},7.3e+03,4.1e+00,4.5e-01,\textbf{4.5e-01},\textbf{4.5e-01},1,0,0,0,0,0,1,0,0
Pothen/mesh2e1,306,30,2.3e-08,1.1e-11,1.7e-11,1.5e-11,2.2e-11,-,23,\textbf{9},\textbf{9},\textbf{9},4.2e+02,1.5e+01,1.8e+00,1.7e+00,\textbf{1.7e+00},7.3e+03,4.1e+00,2.0e-01,\textbf{2.0e-01},2.0e-01,1,0,0,0,0,1,0,0,0
HB/494_bus,494,4,1.8e+00,1.4e-08,6.7e-11,7.4e-11,7.4e-11,-,-,76,\textbf{71}$^\dagger$,\textbf{71}$^\dagger$,3.9e+06,5.8e+04,4.0e+02,\textbf{3.4e+02},\textbf{3.4e+02},-inf,6.9e+01,5.2e+01,\textbf{5.0e+01},\textbf{5.0e+01},1,1,0,0,0,0,1,0,0
HB/494_bus,494,24,1.8e+00,1.4e-08,4.6e-11,6.2e-11,6.2e-11,-,-,42,\textbf{33}$^\dagger$,\textbf{33}$^\dagger$,3.9e+06,5.8e+04,9.8e+01,\textbf{3.9e+01},\textbf{3.9e+01},-inf,6.9e+01,2.8e+01,\textbf{2.2e+01},\textbf{2.2e+01},1,1,0,0,0,0,1,0,0
HB/494_bus,494,49,1.8e+00,1.4e-08,5.1e-11,4.9e-11,4.7e-11,-,-,27,22,\textbf{21},3.9e+06,5.8e+04,3.6e+01,1.2e+01,\textbf{8.6e+00},-inf,6.9e+01,1.5e+01,\textbf{1.2e+01},1.3e+01,1,1,0,0,0,1,0,0,0
HB/662_bus,662,6,1.2e-01,4.9e-11,9.2e-11,9.2e-11,9.2e-11,-,87,\textbf{51}$^\dagger$,\textbf{51}$^\dagger$,\textbf{51}$^\dagger$,8.3e+05,2.7e+04,\textbf{3.7e+02},\textbf{3.7e+02},\textbf{3.7e+02},-inf,6.3e+01,\textbf{4.5e+01},\textbf{4.5e+01},\textbf{4.5e+01},1,0,0,0,0,0,1,1,1
HB/662_bus,662,33,1.2e-01,4.9e-11,4.9e-11,5.0e-11,5.0e-11,-,87,29,\textbf{27}$^\dagger$,\textbf{27}$^\dagger$,8.3e+05,2.7e+04,3.8e+01,\textbf{2.9e+01},3.3e+01,-inf,6.3e+01,2.4e+01,\textbf{2.2e+01},\textbf{2.2e+01},1,0,0,0,0,0,1,0,0
HB/662_bus,662,66,1.2e-01,4.9e-11,4.3e-11,3.5e-11,7.6e-11,-,87,22,20,\textbf{19},8.3e+05,2.7e+04,2.7e+01,1.5e+01,\textbf{1.1e+01},-inf,6.3e+01,1.4e+01,\textbf{1.3e+01},1.3e+01,1,0,0,0,0,1,0,0,0
HB/illc1850,712,7,1.5e-04,2.1e-05,1.6e-05,1.2e-05,1.2e-05,-,-,-,-$^\dagger$,-$^\dagger$,1.4e+07,1.3e+06,7.3e+05,\textbf{5.3e+04},\textbf{5.3e+04},7.4e+02,1.8e+02,1.7e+02,\textbf{1.2e+02},\textbf{1.2e+02},1,1,1,1,1,0,1,0,0
HB/illc1850,712,35,1.5e-04,2.1e-05,9.0e-11,8.6e-11,5.9e-11,-,-,56,\textbf{45},49,1.4e+07,1.3e+06,6.0e+02,\textbf{5.0e+02},5.5e+02,7.4e+02,1.8e+02,6.1e+01,\textbf{5.0e+01},5.5e+01,1,1,0,0,0,0,0,0,0
HB/illc1850,712,71,1.5e-04,2.1e-05,7.4e-11,8.5e-11,8.4e-11,-,-,26,\textbf{24},26,1.4e+07,1.3e+06,\textbf{6.0e+01},6.2e+01,1.2e+02,7.4e+02,1.8e+02,2.5e+01,\textbf{2.2e+01},2.4e+01,1,1,0,0,0,1,0,0,0
HB/bcsstk08,1074,10,3.6e+00,3.0e-11,4.4e-11,3.9e-11,3.9e-11,-,38,27,\textbf{25}$^\dagger$,\textbf{25}$^\dagger$,4.7e+07,8.9e+02,2.9e+01,\textbf{2.5e+01},2.6e+01,-inf,1.7e+01,9.6e+00,\textbf{9.1e+00},\textbf{9.1e+00},1,0,0,0,0,0,1,0,0
HB/bcsstk08,1074,53,3.6e+00,3.0e-11,5.7e-11,3.7e-11,9.4e-11,-,38,\textbf{15},\textbf{15},\textbf{15},4.7e+07,8.9e+02,\textbf{5.8e+00},5.9e+00,6.4e+00,-inf,1.7e+01,2.1e+00,\textbf{2.1e+00},2.2e+00,1,0,0,0,0,0,0,0,0
HB/bcsstk08,1074,107,3.6e+00,3.0e-11,3.7e-11,1.1e-11,1.7e-11,-,38,\textbf{10},\textbf{10},\textbf{10},4.7e+07,8.9e+02,3.1e+00,3.3e+00,\textbf{2.9e+00},-inf,1.7e+01,6.9e-01,\textbf{6.9e-01},6.9e-01,1,0,0,0,0,1,0,0,0
HB/1138_bus,1138,11,1.0e+00,3.5e-02,7.0e-11,7.2e-11,7.2e-11,-,-,77,\textbf{68}$^\dagger$,\textbf{68}$^\dagger$,1.2e+07,2.4e+05,1.0e+03,\textbf{5.3e+02},\textbf{5.3e+02},-inf,1.3e+02,9.5e+01,\textbf{9.0e+01},\textbf{9.0e+01},1,1,0,0,0,0,1,0,0
HB/1138_bus,1138,56,1.0e+00,3.5e-02,7.3e-11,6.6e-11,6.6e-11,-,-,34,\textbf{31}$^\dagger$,\textbf{31}$^\dagger$,1.2e+07,2.4e+05,6.2e+01,\textbf{3.2e+01},3.6e+01,-inf,1.3e+02,4.8e+01,\textbf{4.1e+01},\textbf{4.1e+01},1,1,0,0,0,0,1,0,0
HB/1138_bus,1138,113,1.0e+00,3.5e-02,2.9e-11,8.7e-11,6.5e-11,-,-,25,\textbf{21},\textbf{21},1.2e+07,2.4e+05,2.8e+01,1.9e+01,\textbf{1.0e+01},-inf,1.3e+02,2.6e+01,\textbf{2.2e+01},2.3e+01,1,1,0,0,0,1,0,0,0
LPnetlib/lp_sctap3,1480,14,1.1e-03,1.0e-05,3.2e-07,8.9e-08,8.9e-08,-,-,-,-$^\dagger$,-$^\dagger$,1.4e+05,9.0e+03,6.2e+03,\textbf{5.4e+03},\textbf{5.4e+03},-inf,6.0e+02,5.4e+02,\textbf{5.2e+02},\textbf{5.2e+02},1,1,1,1,1,0,1,0,0
LPnetlib/lp_sctap3,1480,74,1.1e-03,1.0e-05,8.6e-11,9.7e-11,9.7e-11,-,-,\textbf{84},\textbf{84}$^\dagger$,\textbf{84}$^\dagger$,1.4e+05,9.0e+03,\textbf{1.0e+03},1.0e+03,1.0e+03,-inf,6.0e+02,3.0e+02,\textbf{2.8e+02},\textbf{2.8e+02},1,1,0,0,0,0,1,0,0
LPnetlib/lp_sctap3,1480,148,1.1e-03,1.0e-05,6.2e-11,9.6e-11,9.6e-11,-,-,57,\textbf{37}$^\dagger$,\textbf{37}$^\dagger$,1.4e+05,9.0e+03,2.7e+02,\textbf{1.2e+02},1.3e+02,-inf,6.0e+02,9.9e+01,\textbf{6.3e+01},\textbf{6.3e+01},1,1,0,0,0,1,1,0,0
\end{filecontents*}
\csvreader[
head to column names,
after line=\ifthenelse{\equal{\switch}{0}}{\\}{\\\hline},
]{csvs/small/results_out.csv}{Name=\Name}{{\expUScore{\Name}} & \n & \r & \iternopc & \iterichol & \iterrbreg & \itersvd & \iterbreg & \condrbreg & \condsvd & \condbreg & \divrbreg & \divsvd  & \divbreg}
\\[-\normalbaselineskip]\hline
\end{tabular*}
\end{center}
\end{table}

\Cref{table:ichol_results} shows the condition number, Bregman divergence and the number of
PCG iterations at termination, which either occurs if the relative residual is below the tolerance 
$10^{-10}$ or after 100 iterations. The entries "-" in the table correspond to instances
where the tolerance was not reached, either due to PCG stagnating or the maximum number of 
iterations was reached. The $\dagger$ symbol means that the associated
low-rank truncations coincide. 
Since we consider only matrices of small size, we do 
not report on timings since $\PrecondSVD{S}$, $\PrecondBreg{S}$ and $\PrecondBregRev{S}$ are
all constructed from a full eigendecomposition of $S$ and incur similar costs of application
through $Q$ and the Woodbury matrix identity. In \cref{sec:numerical_results:large}, we 
consider larger matrices and report revelant timings.\\

\Cref{table:ichol_results} shows that preconditioning is essential, since unpreconditioned
CG fails to converge in all cases. There are also several matrices for which
convergence is not obtained
using incomplete Cholesky. As expected, the table shows that incorporating a low-rank matrix
can be beneficial for convergence, and increasing the rank results in fewer iterations for
convergence required by all preconditioners.\\

We observe a few trends for the preconditioners $\PrecondBreg{S}$, $\PrecondBregRev{S}$
and $\PrecondSVD{S}$. Overall, $\PrecondBreg{S}$ and $\PrecondBregRev{S}$ are no worse than
$\PrecondSVD{S}$ in terms of iterations, and in most cases improve (albeit sometimes only modestly)
on the TSVD-based preconditioner. The preconditioners $\PrecondBreg{S}$ and $\PrecondBregRev{S}$
perform similarly, and are generally only a few iterations from converging in the same number
of iterations. We observe that while the three truncations can differ, they can lead
to the same results (e.g. \texttt{HB/bcsstk08} with larger $r$). We also observe that 
the preconditioner that yields the smallest preconditioned $\kappa_2$ condition number
of the preconditioned matrix is not always the one for which PCG converges
in the fewest number of steps. However, the divergence between the preconditioner and the 
matrix $S$ (or vice versa) correlates with the fewest iterations required (possibly joint fewest).
These results are not conclusive, and we do not claim that there are no cases where
$\PrecondSVD{S}$ outperforms $\PrecondBreg{S}$ and $\PrecondBregRev{S}$, but could suggest
that the Bregman truncations can provide a better valuation of the trade-off between positive
and negative eigenvalues of $\tilde E$ in terms of PCG convergence.\\

We now have a closer look at a few select instances from \cref{table:ichol_results}.
We see that the three preconditioners $\PrecondBreg{S}$, $\PrecondBregRev{S}$ and
$\PrecondSVD{S}$ appear to yield similar results for the matrix \texttt{HB/lund\_a}. 
Looking at \texttt{HB/lund\_a} for $r=7$ in particular,
\cref{fig:lund_a} shows the eigenvalues of error $\tilde E$, convergence of the relative
residual, and two curves corresponding to the eigenvalues on the Bregman and SVD curves
(i.e. \cref{eq:bregman_curve} and the absolute value, respectively) for $\PrecondBreg{S}$ 
and $\PrecondSVD{S}$. Although there is some overlap in the eigenvalues
selected by the two associated truncations, we also see a few differences. However, due to the spectrum of $\tilde E$ being
clustered around zero, the different truncations contribute more or less equally to the
Bregman log determinant divergence in \cref{eq:Divergence:U=V}, which is also clear from the
values of the divergences in \cref{table:ichol_results}. As a result,
changing $r$ does not result in large deviations between the two preconditioners. Indeed,
the preconditioners $\PrecondBreg{S}$, $\PrecondBregRev{S}$ and $\PrecondSVD{S}$ sometimes
coincide for this example. We notice from the bottom left panel of \cref{fig:lund_a} 
that the rank of $\tilde E$ is also smaller than $n$, with many zero eigenvalues,
making it more likely that the truncations coincide because they have fewer eigenvalues
on which to disagree. Furthermore, the curve $\gamma$ in \cref{eq:bregman_curve} used
to define the BLD  truncation is almost symmetric about zero (and similarly
for the $\nu$ in \cref{eq:bregman_curve:inv} for the RBLD), which also suggests
that the truncations will be similar as for a TSVD.\\

\begin{figure}
    \centering
    \includegraphics[scale=0.39]{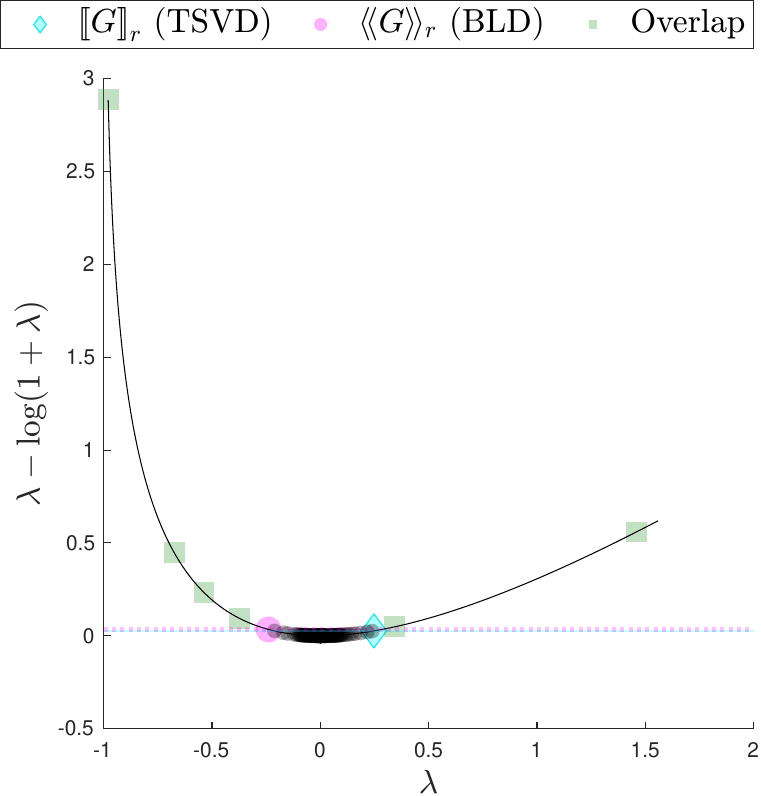}
    \includegraphics[scale=0.39]{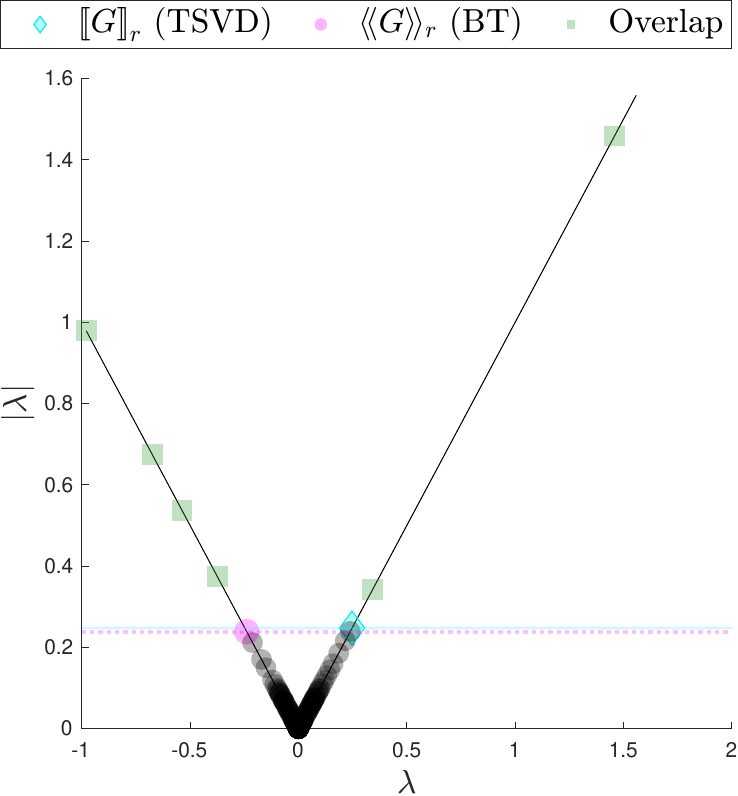}\\
    \includegraphics[scale=0.39]{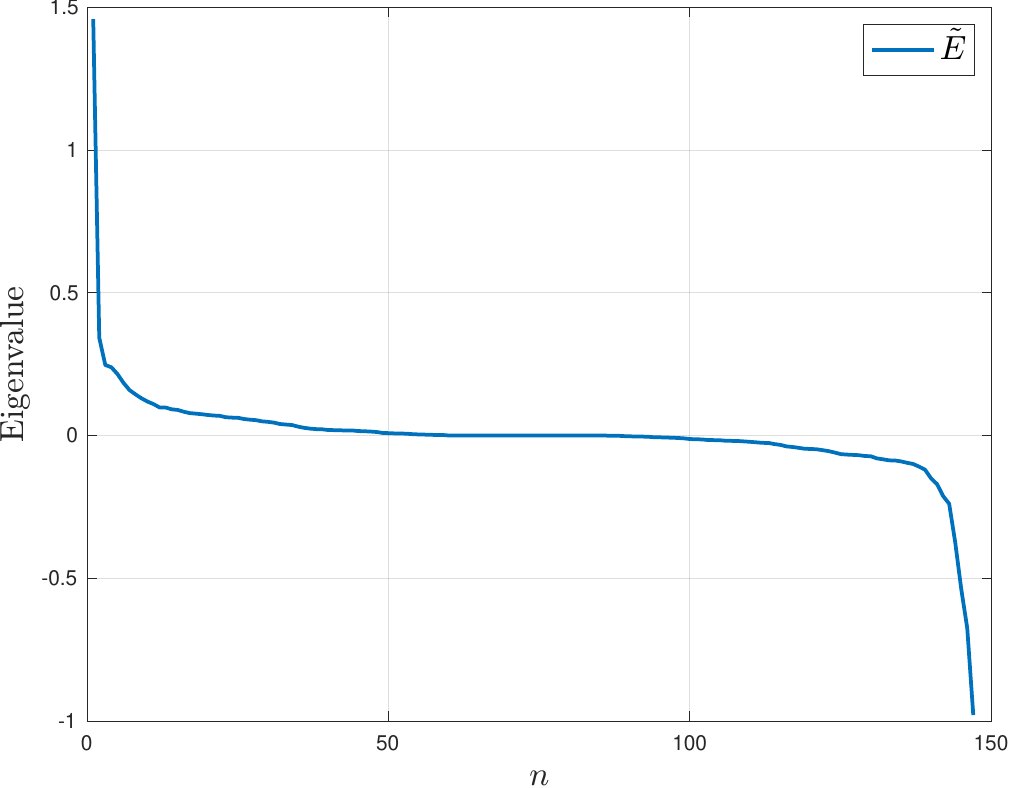}
    \includegraphics[scale=0.39]{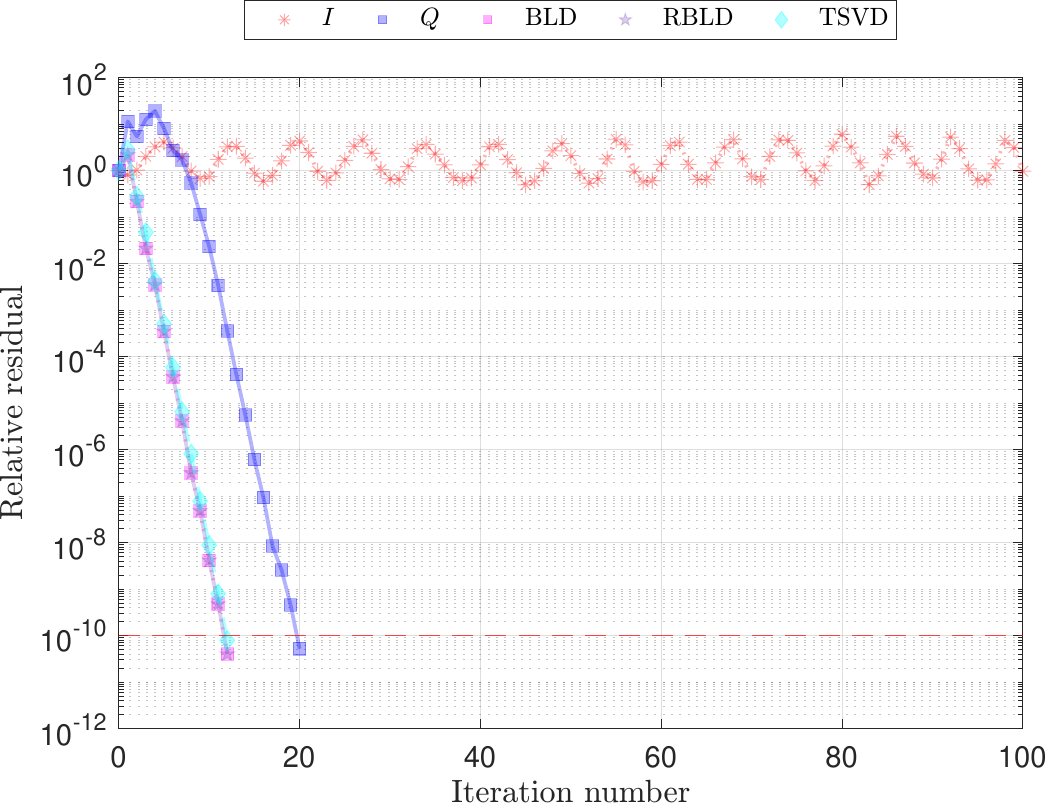}
    \caption{Experiment with \texttt{HB/lund\_a} ($n=147$ and $r=7$). Top: SVD and Bregman curve. Bottom: eigenvalues of $\tilde E$ and PCG convergence.}
    \label{fig:lund_a}
\end{figure}

Next we look at \texttt{HB/illc1850} which is the most ill-conditioned matrix among
the test cases and no convergence to tolerance is observed for $r=7$. Convergence is
achieved as $r$ is increased, with the $\PrecondBreg{S}$ converging in the fewest
amount of iterations. \Cref{fig:illc1850} shows which eigenvalues are included in
the truncations $\BregTrunc{\tilde E}$ and $\tsvd{\tilde E}{r}$. While there is
some overlap, this case nicely illustrates the preferences of the two truncations;
a TSVD prefers eigenvalues of $\tilde E$ of large magnitude, while the BLD truncation
includes more negative eigenvalues, even if they are relatively small in magnitude. The truncations
differ primarily since $\tilde E$ has many positive eigenvalues that are larger than
$1$ in magnitude, and we see a clear difference in PCG convergence between the various 
preconditioners (bottom panels of \cref{fig:illc1850}). As $r$ increases, so does
the similarity between the truncations, and therefore the preconditioners.
For brevity, we do not show the eigenvalues of $\BregTruncRev{\tilde E}$,
although from \cref{fig:bregman_curve:inv} it is clear that this truncation  
appraises negative eigenvalues similarly to $\BregTrunc{\tilde E}$.\\

\begin{figure}
    \centering
    \includegraphics[scale=0.39]{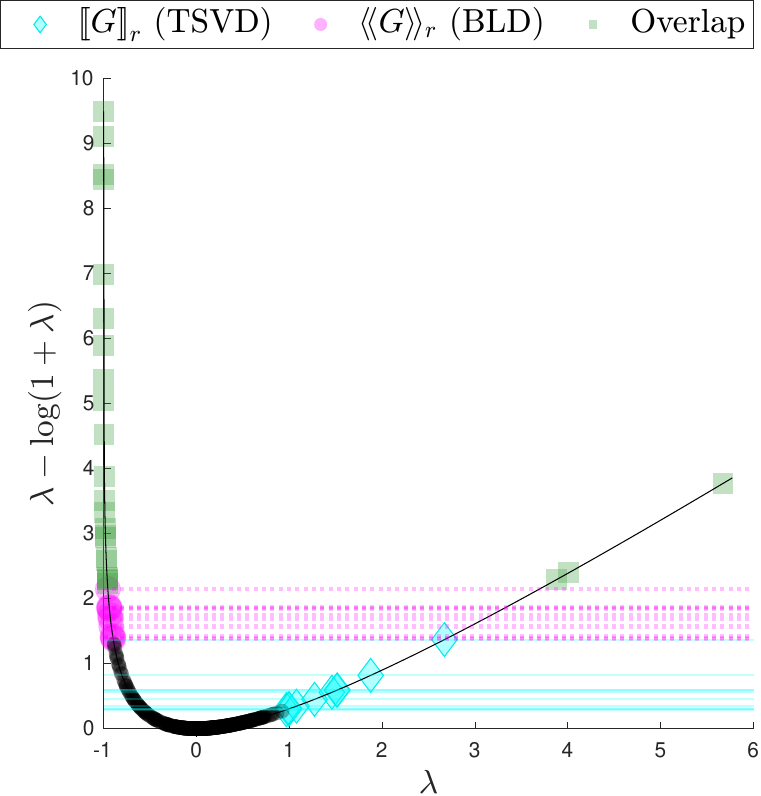}
    \includegraphics[scale=0.39]{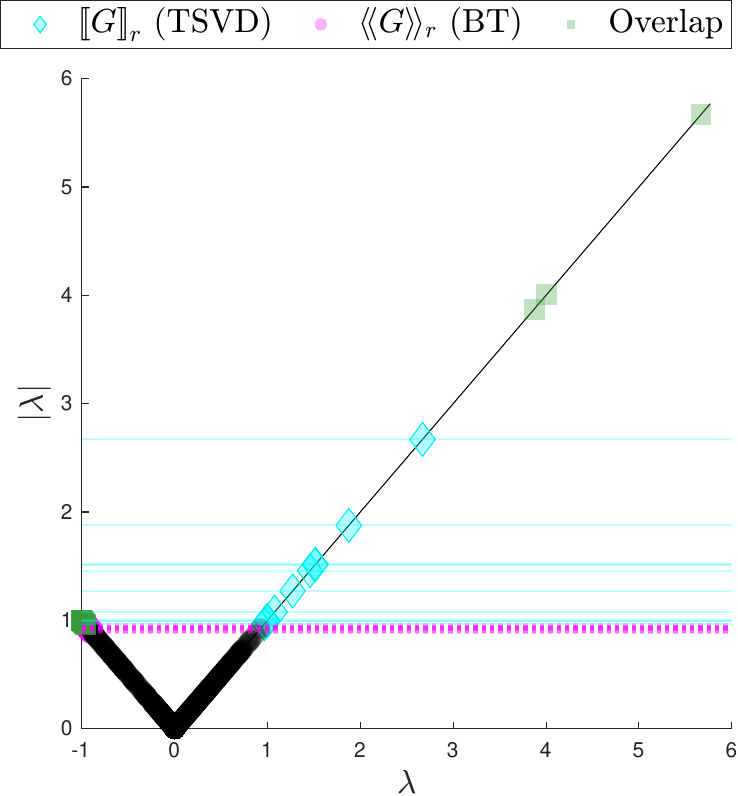}\\
    \includegraphics[scale=0.32]{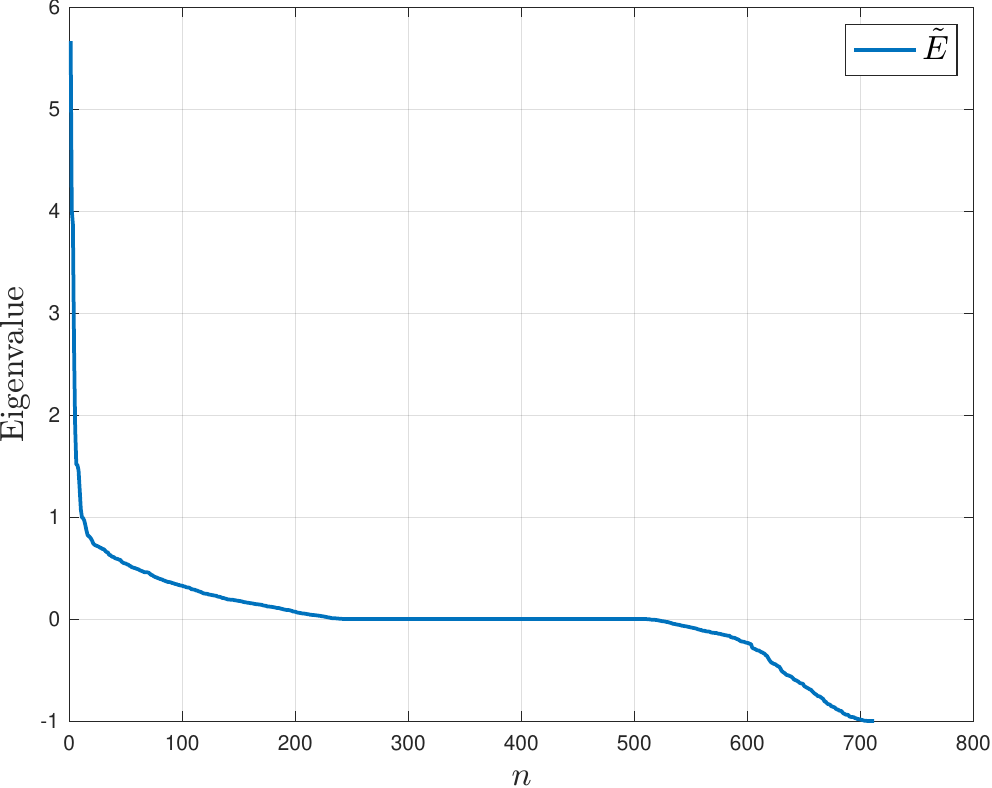}
    \includegraphics[scale=0.35]{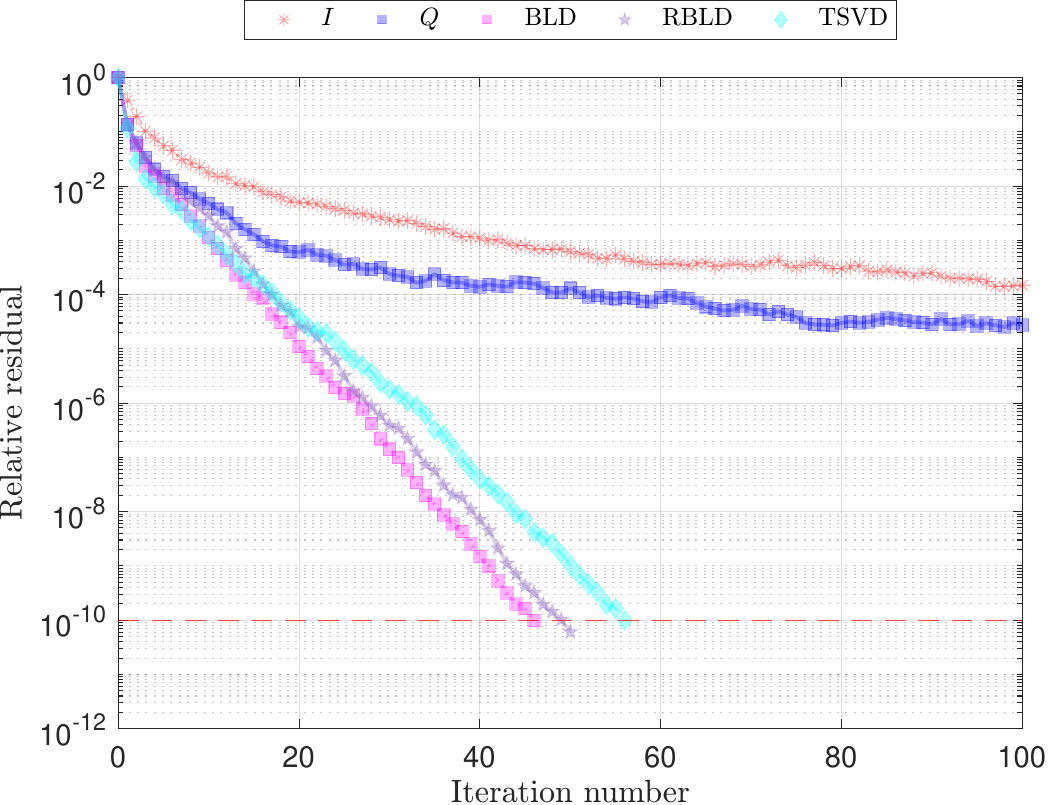}
    \caption{Experiment with \texttt{HB/illc1850} ($n=712$ and $r=35$). Top: SVD and Bregman curve. Bottom: eigenvalues of $\tilde E$ and PCG convergence.}
    \label{fig:illc1850}
\end{figure}

Overall, the results suggest that the novel truncations in
\cref{def:Bregman_truncation,def:Bregman_truncation:inverse} may capture some desired
nearness between a matrix and its preconditioner.
We emphasise that the difference between the two preconditioners can vanish as $r$ increases,
depending on $n$ and the spectrum of $\tilde E$. However, in large-scale applications, 
$r$ may be relatively small compared to $n$, so $\PrecondBreg{S}$ or $\PrecondBregRev{S}$ could
potentially be quite different to $\PrecondSVD{S}$. This could indicate that there can be noticeable
rewards in terms of PCG convergence for judiciously choosing which eigenvalues to choose in a 
truncation. We stress that there is need for theory to support this.

\subsection{Results for larger instances}\label{sec:numerical_results:large}

The exact truncations used in the previous section are typically out of
reach for large $n$. In this section, we experiment with the larger matrices
in \cref{table:ichol_problems:large} and examine the effectiveness of the 
approximations described in \cref{sec:Bregman_directions}.
As in the previous section, $Q$ is found as a zero fill incomplete
Cholesky factorisation of the target matrix $S$. We acknowledge that 
there may exist other incomplete Cholesky factorisations (e.g. by 
setting a non-zero "drop tolerance" as in \MATLAB's implementation,
at the cost of a denser factor) or other bespoke approaches that 
result in improved convergence of PCG. However, it is not always possible
to predict the storage required for a certain drop tolerance and the
consequence (hopefully, reduction) in PCG iterations. As mentioned before,
it is not known what rank $r$ 
produces the right trade-off in terms of computational overhead and
gain in PCG convergence, but the benefit of using 
\cref{eq:Bregman_preconditioner:approx} is that storage of the low-rank
terms remains predictable as $r$ increases. The experiments here focus on investigating
how the convergence of PCG behaves as a function of which parts of the
spectrum of $\tilde E$ we approximate via \cref{eq:BregmanApproxAlpha}.
For the same reasons as in \cref{sec:numerical_results:ichol},
we experiment with values of $r$ relative to the size $n$ of the matrix
\begin{equation}\label{eq:rank_variation:large}
r = \lfloor n \times \epsilon \rfloor, \quad \epsilon \in \{ 0.0025,\, 0.0075\}.
\end{equation}

\begin{table}
\caption{Test set of large $n\times n$ SPD matrices from the SuiteSparse Matrix
Collection. The number of non-zero entries in the matrix is denoted by $\texttt{nnz}$.}
\label{table:ichol_problems:large}
\begin{center}
\begin{tabular}{l l c c c c}
\toprule 
Problem & Application & $n$ & $\texttt{nnz}$ & $\kappa_2$\\\midrule
\texttt{GHS\_psdef/crankseg\_1} & Structural & 52804 & 10614210 & 2.2296e+08\\
\texttt{Oberwolfach/gyro\_m} & Model Reduction & 17361 & 340431 & 2.503983e+06\\
\texttt{UTEP/Dubcova1} & 2D/3D Problem & 16129 & 253009 & 9.971199e+02\\
\texttt{GHS\_psdef/apache1} & Structural & 80800 & 542184 & 3.9883e+06\\
\texttt{Rothberg/cfd1} & Fluid Dynamics & 70656 & 1825580 & 1.3351e+06\\
\texttt{Schmid/thermal1} & Thermal & 82654 & 574458 & 4.9625e+05\\\bottomrule
\end{tabular}
\end{center}
\end{table}

As a surrogate for $\BregTrunc{\tilde E}$, we shall use $\BregmanApproxAlpha$
in \cref{eq:BregmanApproxAlpha} for a prescribed set of values $\alpha \in
[0, 1]$ using the approaches mentioned in 
\cref{sec:Bregman_directions:largest,sec:Bregman_directions:smallest}.
We let $\texttt{maxit}_\textnormal{KS}$, and $\texttt{tol}_\textnormal{KS}$
and $p_\textnormal{KS}$ denote the maximum iterations, convergence tolerance
and Krylov subspace dimension \emph{slack}, respectively. That is, we allow a
subspace dimension of $r+p_\textnormal{KS}$. We use
$\texttt{tol}_\textnormal{KS}=10^{-2}$ in all experiments, and initialise the
Krylov method with a random vector. For the smallest
value of $r$ in \cref{eq:rank_variation:large}, we use $\texttt{maxit}_\textnormal{KS}=60 = 
p_\textnormal{KS} = 60$ for all matrices. For the larger value of $r$,
$\texttt{maxit}_\textnormal{KS}=p_\textnormal{KS} = 100$.
For simplicity, these values have been set generously so that the majority eigenvalues
converge to  the tolerance for all the matrices in \cref{table:ichol_problems:large}.
The values of $\texttt{maxit}_\textnormal{KS}$ and $p_\textnormal{KS}$ may be
reduced depending on the matrix, and, in particular, for smaller $r$.\\

$\tsvd{\tilde E}{r}$ will be replaced by either the Nyström approximation
$\Nystrom{\tilde E}$ or the modified Nyström $\NystromIndefinite{\tilde E}$
for indefinite matrices. We define $\Nystrom{\tilde E}$ using a sketching
matrix $\Omega$ of size $n\times (r + p)$, $p=60$ (\cref{eq:Nystrom}). For
$\NystromIndefinite{\tilde E}$, we set $c=1.5$ (\cref{eq:Nystrom:indefinite}).\\

We also introduce
\[
\PrecondSVDKS{S} = Q(I + \tsvdks{\tilde E}{r})Q\Hermitian,
\]
where $\tsvdks{\tilde E}{r} \approx \tsvd{\tilde E}{r}$ has been constructed 
using the \MATLAB's \texttt{eigs} routine mentioned previously and approximates
$\tilde E$ in sense of its largest magnitude eigenvalues. This uses the same
parameters as described above.\\

 To summarise, we compare PCG results for the following preconditioners:
\begin{enumerate}
    \item no preconditioner,
    \item symmetric preconditioning using $Q$,
    \item $\PrecondNys{S} = Q(I + \Nystrom{\tilde E})Q\Hermitian$,
    \item $\PrecondNysIndef{S} = Q(I + \NystromIndefinite{\tilde E})Q\Hermitian$,
    \item $\PrecondSVDKS{S} = Q(I + \tsvdks{\tilde E}{r})Q\Hermitian$,
    \item and $\PrecondBregAlpha{S}{\alpha} = Q(I + \BregmanApproxAlpha)Q\Hermitian$,
    for \[\alpha \in \{0,\, 0.25,\, 0.5,\, 0.75,\, 1\}.\]
\end{enumerate}

\Cref{table:GHS_psdef_crankseg_1,table:Dubcova2,table:Dubcova1,table:GHS_psdef_apache1,table:Rothberg_cfd1,table:Schmid_thermal1} show the relative residual upon termination, the number
of iterations and matrix-vector products with the matrix, as well as construction
and solve times. The convergence tolerance is set to $10^{-10}$ after a maximum of
350 iterations.\\


Foremost, we see that neither of the Nyström-based preconditioners $\PrecondNys{S}$
and $\PrecondNysIndef{S}$ perform well. We cannot expect these randomised approaches to
yield highly accurate approximations of $\tilde E$, but they are
significantly cheaper to construct in terms of both time and matrix-vector products.
If the setting allows for a higher rank $r$, we may expect to see improved results
using such randomised approaches.\\

In \cref{table:GHS_psdef_crankseg_1,table:Dubcova1,table:Dubcova2}, the preconditioner
$\PrecondBregAlpha{S}{\alpha}$ results in fewer or as many PCG iterations than 
 $\PrecondSVDKS{S}$ depending on $\alpha$. There are also cases where the opposite
 is true, and generally, we see that PCG convergence
 deteriorates as $\alpha$ increases. This may suggest that, in some cases,
 capturing the eigensubspace of the error $\tilde E$ corresponding to the
smallest eigenvalues can be valuable over the that with largest singular values, which 
aligns well with our observations from \cref{sec:numerical_results:ichol}.
There are also several cases where the opposite is true and $\PrecondSVDKS{S}$
outperforms $\PrecondBregAlpha{S}{\alpha}$. Allowing a more fine-grained parameter
search of the variable $\alpha$ may be illuminating, to investigate whether there
exists a value $\alpha$ for which $\PrecondBregAlpha{S}{\alpha}$ performs similarly
to $\PrecondSVDKS{S}$. The question would then be whether this corresponds to the
case where
\[
\PrecondBregAlpha{S}{\alpha} \approx \PrecondSVDKS{S}.
\]
The results, however, suggest that there may be significant gains in allocating 
computational budget towards other truncations than a TSVD, although it is difficult
to determine an appropriate value of $\alpha$ in such cases.\\

Overall, the preconditioners $\PrecondBregAlpha{S}{\alpha}$ and $\PrecondSVDKS{S}$ are
expensive to compute, but vastly outperforms the randomised counterparts.
We see that many matrix-vector products are required for the convergence of the
Krylov--Schur method used to construct $\PrecondSVDKS{S}$ and $\PrecondBregAlpha{S}{\alpha}$. Although $\tilde E$ may be a dense matrix (but is never formed), the
application of $\tilde E$ to a vector may be cheap if both triangular solves with
$Q$ can be done efficiently and the matrix $S$ is sparse. The computation of 
$\PrecondBregAlpha{S}{\alpha}$ can therefore be fast for certain instances.
Construction time of the $\PrecondBregAlpha{S}{\alpha}$ increases with the number
of required matrix-vector products, which for small $\alpha$ generally is larger than
if $\alpha$ is small. There is an obvious trade-off between PCG convergence and the
time spent on the construction of $\PrecondBregAlpha{S}{\alpha}$, although the 
interplay between $r$ and $\alpha$ is not known in general.
If a given application requires multiple solves with the same matrix, then the
cost of construction of the preconditioner $\PrecondBregAlpha{S}{\alpha}$ is amortised
and therefore may be an attractive option.\\

Empirically, it appears
that $\PrecondBregAlpha{S}{\alpha}$
can reduce the number of PCG iterations, but this depends on the spectrum of $\tilde E$.
Knowledge of the spectrum, together with the overall computational budget (rank $r$,
fill of $Q$, cost of matrix-vector products and access model of $S$), informs the
value of $\alpha$.

{\renewcommand{\arraystretch}{1.0}
\begin{table}
\scriptsize
\caption{Results for \texttt{GHS/psdef\_crankseg\_1}.}\label{table:GHS_psdef_crankseg_1}
\begin{center}
\begin{tabular*}{0.73\textwidth}{l @{\extracolsep{\fill}} *{8}{c}}
\toprule 
Preconditioner & $r$ & $\alpha$ & Rel. residual & Iteration count & Construction ($s$) & Solve ($s$) & Matvecs\\
\midrule
$I$ & - & - & 1.4e-01 & - & - & 6.7e+00 & 0\\\CSVICHOL & - & - & 9.3e-11 & 183 & 1.26 & 7.7e+00 & 0\\\CSVSVDKS & 132 & - & 7.6e-11 & 158 & 72.9 & 1.7e+01 & 672\\\multirow{5}{*}{$\CSVPrecondBregAlpha{lpha}$} & 132 & 0.00 & 8.6e-11 & 77 & 35.9 & 8.4e+00 & 714\\ & 132 & 0.25 & 8.6e-11 & 77 & 28.1 & 8.7e+00 & 583\\ & 132 & 0.50 & 8.3e-11 & 74 & 36.6 & 8.1e+00 & 753\\ & 132 & 0.75 & 7.6e-11 & 76 & 41.0 & 8.3e+00 & 846\\ & 132 & 1.00 & 8.7e-11 & 158 & 33.4 & 1.7e+01 & 642\\$\CSVPrecondNys$ & 132 & - & 1.0e-05 & - & 9.17 & 3.8e+01 & 132\\$\CSVPrecondNysIndef$ & 132 & - & 7.1e-05 & - & 7.69 & 3.8e+01 & 198\\\CSVSVDKS & 396 & - & 8.0e-11 & 152 & 309.0 & 6.8e+01 & 1996\\\multirow{5}{*}{$\CSVPrecondBregAlpha{lpha}$} & 396 & 0.00 & 9.7e-11 & 63 & 102.0 & 2.8e+01 & 1542\\ & 396 & 0.25 & 8.3e-11 & 60 & 78.1 & 2.7e+01 & 1292\\ & 396 & 0.50 & 7.9e-11 & 61 & 98.1 & 2.8e+01 & 1610\\ & 396 & 0.75 & 9.3e-11 & 65 & 113.0 & 2.9e+01 & 1792\\ & 396 & 1.00 & 8.5e-11 & 154 & 94.8 & 6.9e+01 & 1296\\$\CSVPrecondNys$ & 396 & - & 5.7e-03 & - & 19.9 & 1.6e+02 & 396\\$\CSVPrecondNysIndef$ & 396 & - & 2.0e-03 & - & 24.1 & 1.6e+02 & 594\\
\\[-\normalbaselineskip]\bottomrule
\end{tabular*}
\end{center}
\end{table}
}

{\renewcommand{\arraystretch}{1.0}
\begin{table}
\scriptsize
\caption{Results for \texttt{UTEP/Dubcova1}.}\label{table:Dubcova1}
\begin{center}
\begin{tabular*}{0.73\textwidth}{l @{\extracolsep{\fill}} *{8}{c}}
\toprule 
Preconditioner & $r$ & $\alpha$ & Rel. residual & Iteration count & Construction ($s$) & Solve ($s$) & Matvecs\\
\midrule
$I$ & - & - & 9.6e-11 & 175 & - & 1.6e-01 & 0\\\CSVICHOL & - & - & 9.0e-11 & 94 & 0.00735 & 1.5e-01 & 0\\\CSVSVDKS & 40 & - & 7.4e-11 & 50 & 3.36 & 3.1e-01 & 389\\\multirow{5}{*}{$\CSVPrecondBregAlpha{lpha}$} & 40 & 0.00 & 9.0e-11 & 33 & 1.12 & 3.2e-01 & 372\\ & 40 & 0.25 & 8.0e-11 & 38 & 1.36 & 2.3e-01 & 381\\ & 40 & 0.50 & 7.4e-11 & 45 & 0.892 & 2.3e-01 & 338\\ & 40 & 0.75 & 9.9e-11 & 55 & 1.17 & 3.1e-01 & 477\\ & 40 & 1.00 & 9.8e-11 & 93 & 1.0 & 4.7e-01 & 349\\$\CSVPrecondNys$ & 40 & - & 8.5e-11 & 204 & 0.319 & 1.5e+00 & 40\\$\CSVPrecondNysIndef$ & 40 & - & 9.7e-11 & 254 & 0.092 & 1.7e+00 & 60\\\CSVSVDKS & 120 & - & 8.4e-11 & 37 & 11.3 & 7.3e-01 & 720\\\multirow{5}{*}{$\CSVPrecondBregAlpha{lpha}$} & 120 & 0.00 & 8.6e-11 & 27 & 3.89 & 5.1e-01 & 849\\ & 120 & 0.25 & 8.4e-11 & 27 & 3.48 & 5.3e-01 & 821\\ & 120 & 0.50 & 5.1e-11 & 29 & 3.65 & 5.9e-01 & 900\\ & 120 & 0.75 & 6.7e-11 & 36 & 3.35 & 7.1e-01 & 810\\ & 120 & 1.00 & 9.9e-11 & 91 & 4.26 & 1.8e+00 & 840\\$\CSVPrecondNys$ & 120 & - & 2.8e-09 & - & 0.299 & 6.7e+00 & 120\\$\CSVPrecondNysIndef$ & 120 & - & 8.1e-08 & - & 0.244 & 7.4e+00 & 180\\
\\[-\normalbaselineskip]\bottomrule
\end{tabular*}
\end{center}
\end{table}
}

{\renewcommand{\arraystretch}{1.0}
\begin{table}
\scriptsize
\caption{Results for \texttt{UTEP/Dubcova2}.}\label{table:Dubcova2}
\begin{center}
\begin{tabular*}{0.73\textwidth}{l @{\extracolsep{\fill}} *{8}{c}}
\toprule 
Preconditioner & $r$ & $\alpha$ & Rel. residual & Iteration count & Construction ($s$) & Solve ($s$) & Matvecs\\
\midrule
$I$ & - & - & 9.7e-11 & 340 & - & 1.1e+00 & 0\\\CSVICHOL & - & - & 9.1e-11 & 173 & 0.0259 & 1.6e+00 & 0\\\CSVSVDKS & 162 & - & 8.0e-11 & 47 & 74.1 & 6.4e+00 & 853\\\multirow{5}{*}{$\CSVPrecondBregAlpha{lpha}$} & 162 & 0.00 & 6.8e-11 & 39 & 32.0 & 5.3e+00 & 1108\\ & 162 & 0.25 & 6.3e-11 & 45 & 27.9 & 6.1e+00 & 1123\\ & 162 & 0.50 & 1.0e-10 & 51 & 26.2 & 6.9e+00 & 1115\\ & 162 & 0.75 & 9.1e-11 & 61 & 29.3 & 8.3e+00 & 1228\\ & 162 & 1.00 & 8.5e-11 & 173 & 25.3 & 2.3e+01 & 823\\$\CSVPrecondNys$ & 162 & - & 7.0e-05 & - & 2.61 & 4.8e+01 & 162\\$\CSVPrecondNysIndef$ & 162 & - & 2.4e-04 & - & 1.99 & 4.8e+01 & 243\\\CSVSVDKS & 487 & - & 9.5e-11 & 33 & 365.0 & 2.3e+01 & 1787\\\multirow{5}{*}{$\CSVPrecondBregAlpha{lpha}$} & 487 & 0.00 & 7.8e-11 & 28 & 110.0 & 2.0e+01 & 2024\\ & 487 & 0.25 & 5.3e-11 & 29 & 108.0 & 2.0e+01 & 2299\\ & 487 & 0.50 & 7.0e-11 & 29 & 135.0 & 2.0e+01 & 2806\\ & 487 & 0.75 & 7.6e-11 & 39 & 128.0 & 2.7e+01 & 2631\\ & 487 & 1.00 & 8.2e-11 & 172 & 136.0 & 1.2e+02 & 2037\\$\CSVPrecondNys$ & 487 & - & 7.2e-03 & - & 6.15 & 2.5e+02 & 487\\$\CSVPrecondNysIndef$ & 487 & - & 2.3e-03 & - & 7.32 & 2.5e+02 & 731\\
\\[-\normalbaselineskip]\bottomrule
\end{tabular*}
\end{center}
\end{table}
}

{\renewcommand{\arraystretch}{1.0}
\begin{table}
\scriptsize
\caption{Results for \texttt{GHS/psdef\_apache1.}}\label{table:GHS_psdef_apache1}
\begin{center}
\begin{tabular*}{0.73\textwidth}{l @{\extracolsep{\fill}} *{8}{c}}
\toprule 
Preconditioner & $r$ & $\alpha$ & Rel. residual & Iteration count & Construction ($s$) & Solve ($s$) & Matvecs\\
\midrule
$I$ & - & - & 6.0e-02 & - & - & 6.1e-01 & 0\\\CSVICHOL & - & - & 5.3e-10 & - & 0.00911 & 2.1e+00 & 0\\\CSVSVDKS & 202 & - & 9.2e-11 & 45 & 132.0 & 1.1e+01 & 1312\\\multirow{5}{*}{$\CSVPrecondBregAlpha{lpha}$} & 202 & 0.00 & 8.3e-11 & 53 & 43.8 & 1.3e+01 & 1334\\ & 202 & 0.25 & 7.4e-11 & 64 & 35.3 & 1.5e+01 & 1322\\ & 202 & 0.50 & 9.0e-11 & 76 & 34.2 & 1.8e+01 & 1363\\ & 202 & 0.75 & 9.7e-11 & 90 & 37.2 & 2.1e+01 & 1436\\ & 202 & 1.00 & 5.3e-10 & - & 34.4 & 8.3e+01 & 982\\$\CSVPrecondNys$ & 202 & - & 5.3e-02 & - & 2.34 & 8.5e+01 & 202\\$\CSVPrecondNysIndef$ & 202 & - & 5.0e-02 & - & 2.01 & 8.3e+01 & 303\\\CSVSVDKS & 606 & - & 6.0e-11 & 26 & 718.0 & 3.3e+01 & 2906\\\multirow{5}{*}{$\CSVPrecondBregAlpha{lpha}$} & 606 & 0.00 & 6.6e-11 & 31 & 213.0 & 4.1e+01 & 2862\\ & 606 & 0.25 & 5.7e-11 & 32 & 193.0 & 4.1e+01 & 3203\\ & 606 & 0.50 & 7.9e-11 & 34 & 204.0 & 4.4e+01 & 3501\\ & 606 & 0.75 & 6.0e-11 & 47 & 227.0 & 6.0e+01 & 3739\\ & 606 & 1.00 & 4.0e-10 & - & 214.0 & 4.5e+02 & 2406\\$\CSVPrecondNys$ & 606 & - & 1.6e-01 & - & 7.17 & 4.6e+02 & 606\\$\CSVPrecondNysIndef$ & 606 & - & 1.6e-01 & - & 9.61 & 4.5e+02 & 909\\
\\[-\normalbaselineskip]\bottomrule
\end{tabular*}
\end{center}
\end{table}
}

{\renewcommand{\arraystretch}{1.0}
\begin{table}
\scriptsize
\caption{Results for \texttt{Rothberg/cfd1}.}\label{table:Rothberg_cfd1}
\begin{center}
\begin{tabular*}{0.73\textwidth}{l @{\extracolsep{\fill}} *{8}{c}}
\toprule 
Preconditioner & $r$ & $\alpha$ & Rel. residual & Iteration count & Construction ($s$) & Solve ($s$) & Matvecs\\
\midrule
$I$ & - & - & 5.7e-02 & - & - & 1.5e+00 & 0\\\CSVICHOL & - & - & 1.8e-04 & - & 0.0397 & 4.5e+00 & 0\\\CSVSVDKS & 176 & - & 7.1e-11 & 60 & 104.0 & 1.1e+01 & 1106\\\multirow{5}{*}{$\CSVPrecondBregAlpha{lpha}$} & 176 & 0.00 & 8.3e-11 & 86 & 31.6 & 1.5e+01 & 982\\ & 176 & 0.25 & 8.6e-11 & 88 & 29.1 & 1.6e+01 & 995\\ & 176 & 0.50 & 7.4e-11 & 92 & 29.9 & 1.6e+01 & 1051\\ & 176 & 0.75 & 8.5e-11 & 132 & 32.9 & 2.4e+01 & 1092\\ & 176 & 1.00 & 2.1e-05 & - & 32.8 & 6.3e+01 & 896\\$\CSVPrecondNys$ & 176 & - & 2.8e-02 & - & 4.15 & 6.2e+01 & 176\\$\CSVPrecondNysIndef$ & 176 & - & 2.8e-02 & - & 3.56 & 6.2e+01 & 264\\\CSVSVDKS & 529 & - & 7.5e-11 & 35 & 618.0 & 3.4e+01 & 3179\\\multirow{5}{*}{$\CSVPrecondBregAlpha{lpha}$} & 529 & 0.00 & 7.2e-11 & 59 & 139.0 & 5.8e+01 & 2158\\ & 529 & 0.25 & 8.4e-11 & 56 & 116.0 & 5.5e+01 & 2206\\ & 529 & 0.50 & 8.2e-11 & 61 & 125.0 & 6.0e+01 & 2369\\ & 529 & 0.75 & 9.8e-11 & 74 & 122.0 & 7.2e+01 & 2194\\ & 529 & 1.00 & 8.3e-06 & - & 122.0 & 3.4e+02 & 1629\\$\CSVPrecondNys$ & 529 & - & 1.2e-01 & - & 9.84 & 3.5e+02 & 529\\$\CSVPrecondNysIndef$ & 529 & - & 1.8e-01 & - & 12.3 & 3.5e+02 & 794\\
\\[-\normalbaselineskip]\bottomrule
\end{tabular*}
\end{center}
\end{table}
}

{\renewcommand{\arraystretch}{1.0}
\begin{table}
\scriptsize
\caption{Results for \texttt{Schmid/thermal1.}}\label{table:Schmid_thermal1}
\begin{center}
\begin{tabular*}{0.73\textwidth}{l @{\extracolsep{\fill}} *{8}{c}}
\toprule 
Preconditioner & $r$ & $\alpha$ & Rel. residual & Iteration count & Construction ($s$) & Solve ($s$) & Matvecs\\
\midrule
$I$ & - & - & 1.9e-02 & - & - & 8.2e-01 & 0\\\CSVICHOL & - & - & 6.9e-04 & - & 0.0158 & 2.5e+00 & 0\\\CSVSVDKS & 206 & - & 8.6e-11 & 121 & 140.0 & 3.1e+01 & 1316\\\multirow{5}{*}{$\CSVPrecondBregAlpha{lpha}$} & 206 & 0.00 & 8.7e-11 & 151 & 42.8 & 3.9e+01 & 1282\\ & 206 & 0.25 & 9.6e-11 & 177 & 37.1 & 4.6e+01 & 1339\\ & 206 & 0.50 & 9.3e-11 & 230 & 31.4 & 5.8e+01 & 1214\\ & 206 & 0.75 & 9.5e-11 & 324 & 32.0 & 8.3e+01 & 1170\\ & 206 & 1.00 & 4.6e-04 & - & 31.9 & 8.9e+01 & 896\\$\CSVPrecondNys$ & 206 & - & 1.3e-01 & - & 2.88 & 9.0e+01 & 206\\$\CSVPrecondNysIndef$ & 206 & - & 1.2e-01 & - & 2.37 & 8.9e+01 & 309\\\CSVSVDKS & 619 & - & 9.3e-11 & 52 & 800.0 & 7.1e+01 & 3219\\\multirow{5}{*}{$\CSVPrecondBregAlpha{lpha}$} & 619 & 0.00 & 7.9e-11 & 78 & 219.0 & 1.1e+02 & 2838\\ & 619 & 0.25 & 9.9e-11 & 91 & 146.0 & 1.3e+02 & 2454\\ & 619 & 0.50 & 9.8e-11 & 111 & 155.0 & 1.5e+02 & 2722\\ & 619 & 0.75 & 9.0e-11 & 146 & 161.0 & 2.0e+02 & 2747\\ & 619 & 1.00 & 4.2e-04 & - & 165.0 & 4.8e+02 & 1919\\$\CSVPrecondNys$ & 619 & - & 1.1e-01 & - & 8.33 & 4.9e+02 & 619\\$\CSVPrecondNysIndef$ & 619 & - & 9.2e-02 & - & 10.9 & 4.9e+02 & 929\\
\\[-\normalbaselineskip]\bottomrule
\end{tabular*}
\end{center}
\end{table}
}

\section{Summary \& outlook}\label{sec:summary}

We have presented preconditioners based on low-rank truncations constructed as
minimisers of the Bregman log determinant divergence.
The motivation for this work was to systematically account for factorisation errors 
$S - QQ\Hermitian$ when designing preconditioners for the iterative solution of
linear systems \cref{eq:Sxb} involving the matrix $S$.
The proposed preconditioners are intended for "black box" computations, i.e. they do
not rely on any particular structure of $S$.\\

In \cref{sec:approximation_problem}, we showed how the Bregman log determinant divergence
generally leads to principal directions of a matrix that are different from a TSVD, and
\cref{sec:preconditioner} presented the associated preconditioner. We proposed an approximation
of the BLD truncation in \cref{sec:Bregman_directions} based on a combination of the
Nyström approximation and a Krylov--Schur method and analysed the resulting preconditioner
in \cref{sec:preconditioner:approx}.
\Cref{sec:numerical_results} contained numerous numerical examples comparing the exact
BLD truncations to an TSVD-based preconditioner (\cref{sec:numerical_results:ichol})
as well as demonstrating the effectiveness of approximations to our preconditioner in
\cref{sec:numerical_results:large}.\\

Our results depend heavily on the choice of factor $Q$ and how closely
$QQ\Hermitian$ resembles the target matrix $S$ (cf. the discussion in 
\cref{sec:preconditioner:approx}), and approximate factorisations come in many 
variants \cite{scott2023algorithms}. For instance, off-diagonal entries of a
Cholesky factorisation that are below a certain threshold can be set to zero in
the factorisation to force a degree of sparsity. An advantage of the framework 
is that while storage costs of $Q$ may be unpredictable as a function of drop 
tolerance, the $nr$ storage of the low-rank compensation term is deterministic.
For a given problem, it may be the case that specialised procedures alleviate
the need for compensation using a low-rank matrix. However, this may not be true
in general, and our results provide a novel framework for this should it be necessary.\\

In any case, it can be desirable to
construct $Q$ in a way that ensures positive semi-definiteness of $\tilde E$
leading to the preconditioner $\PrecondSVD{S}$ conceived in 
\cite{bock2023preconditioner}, which is much easier to approximate using
the Nyström approximation when $\tilde E$ is positive semi-definite. Another advantage
is that sketching can easily be parallelised, whereas it is not obvious to do so
for the Krylov--Schur method adopted above.
Investigating sparse approximate inverse preconditioners 
\cite{benzi1999comparative} could also be valuable in this context.
This relates to a major disadvantage of the proposed preconditioner. An
appropriate choice of rank $r$ that balances computational cost of the 
associated truncations with any gain related to convergence of PCG is not
known beforehand, and, as the experiments, show depends heavily on the 
spectrum of $\tilde E$, which in turn depends on $Q$. A thorough 
computational study comparing the preconditioners above to other
preconditioners in the literature is also necessary. \\

Theory is necessary to fully understand the nature of the truncations,
connections to PCG, and the results in \cref{sec:numerical_results}.
The experiments indicate that it is not always preferable to select 
truncations based on the largest singular values when considering preconditioners
for iterative solvers. One objective would be to understand the difference
between $\PrecondBreg{S}$ and $\PrecondBregRev{S}$, which is given by their
respective truncations of the error $\tilde E$. Using similar approaches
to \cite{higham2019new}, we leave a more detailed theoretical and numerical
study on how the choice of $Q$ affects the proposed preconditioner as future
work.\\

As alluded to in \cref{remark:exact}, investigating the effect of permutations
on the preconditioner in line with \cite{duff1989effect} could also be interesting. 
When $Q = \mathcal{P}\inv C$ is defined in terms of a fill-reducing permutation,
a reduction of the rank of $S - \mathcal{P}\inv C C\Hermitian \mathcal{P}\invtransp$
could be sought as part of the approximate factorisation. We can express this
problem as follows, where $\mathcal{S}_n$ is the group of symmetric permutations:
\begin{equation}\label{eq:P_Sn_problem}
\minimise_{P\in \mathcal{S}_n}\quad \rank(S - \mathcal{P}\inv C C\Hermitian \mathcal{P}\invtransp).
\end{equation}
Using the Ky-Fan norm, we can pose a convex relaxation of \cref{eq:P_Sn_problem}
\begin{align*}
\minimise_{P\in \mathcal{S}_n}\quad & \sum_{i=1}^n \sigma_i(S - \mathcal{P}\inv C C\Hermitian \mathcal{P}\invtransp),
\end{align*}
the solution of which is subject to future research. Chordal completions could also
be relevant to the  problem at hand. Suppose $QQ\Hermitian$ has a chordal
sparsity pattern given by the adjacency matrix $E$, then there exists a 
perturbation matrix $P$ for which $P A P\transp$ has a zero fill Cholesky 
factorisation. Finding minimal fill of a factorisation can therefore be 
viewed as a minimum chordal completion problem \cite{vandenberghe2015chordal}.

\section{Acknowledgements}
We thank the anonymous referees, whose comments and suggestions greatly improved this paper.
This work was supported by the Novo Nordisk Foundation under grant number NNF20OC0061894. 


\bibliographystyle{siam}
\bibliography{main}

\end{document}